\documentclass[12pt]{article}

\usepackage{graphicx}

\usepackage{authblk}

\usepackage{tikz}
\usetikzlibrary{arrows,calc}

\tikzstyle{arc}=[->, shorten <=3pt, shorten >=3pt, >=stealth, line width=1.5pt]
\tikzstyle{edge}=[shorten <=2pt, shorten >=2pt, >=stealth, line width=1.5pt]
\tikzstyle{vertex}=[circle, fill=white, draw, minimum size=6pt, inner sep=0pt,
                    outer sep=0pt]
\tikzstyle{blackV}=[circle, fill=black, draw, minimum size=6pt, inner sep=0pt,
                    outer sep=0pt]

\usepackage{float}

\usepackage{xcolor}

\usepackage{amsmath}
\usepackage{amssymb}

\usepackage{amsthm}

\usepackage[capitalize]{cleveref}

\usepackage{multicol}

\usepackage[linesnumbered,ruled,vlined]{algorithm2e}

\usepackage{bbm}

\usepackage{fullpage}

\newtheorem{theorem}{Theorem}
\newtheorem{lemma}[theorem]{Lemma}
\newtheorem{corollary}[theorem]{Corollary}
\newtheorem{proposition}[theorem]{Proposition}
\newtheorem{problem}[theorem]{Problem}
\newtheorem{observation}[theorem]{Observation}

\usepackage{tikz}
\usetikzlibrary{arrows,shapes,matrix,decorations.pathmorphing}

\tikzstyle{arc}=[->,shorten <=3pt, shorten >=3pt,
                 >=stealth, line width=1.1pt]
\tikzstyle{edge}=[shorten <=2pt, shorten >=2pt,
                  >=stealth, line width=1.1pt]
\tikzstyle{vertex}=[circle, fill=white, draw,
                    minimum size=7pt,
                    inner sep=0pt, outer sep=0pt]

\usepackage{url}

\title{Orientations without forbidden patterns on three vertices%
\thanks{This research was supported by grants  UNAM-PAPIIT IA101423,
SEP-CONACYT A1-S-8397, and CONACYT FORDECYT-PRONACES/39570/2020.}}

\author[1]{Santiago Guzm\'an-Pro\thanks{sanguzpro@ciencias.unam.mx}}
\author[1]{C\'esar Hern\'andez-Cruz\thanks{chc@ciencias.unam.mx}}

\affil[2]{Facultad de Ciencias\\
          Universidad Nacional Aut\'onoma de M\'exico\\
          Av. Universidad 3000, Circuito Exterior S/N\\
          C.P. 04510, Ciudad Universitaria, CDMX, M\'exico}

\begin{document}
\date{}

\maketitle

\begin{abstract}
Given a set $F$ of oriented graphs, a graph $G$ is a $Forb_e(F)$-graph if it
admits an $F$-free orientation. Skrien showed that proper-circular arc graphs,
nested interval graphs and comparability graphs, correspond to $Forb_e(F)$-graph
classes for some set $F$ of orientations of $P_3$. Building on these results, we
exhibit the list of all $Forb_e(F)$-graph classes when $F$ is a set of oriented
graphs on three vertices. Structural characterizations for these classes are
provided, except for the so-called perfectly-orientable graphs and the
transitive-perfectly-orientable graphs, which remain as open problems.
\end{abstract}

\section{Introduction}
\label{sec:Introduction}

Given a set $F$ of oriented graphs, Skrien defined the class of $F$-graphs to
be the class of graphs that admit an $F$-free orientation \cite{skrienJGT6}. As
stated in \cite{guzmanEJC105}, we believe that this definition might be
misleading in the sense that the class of $F$-graphs is negatively defined with
respect to $F$. For this  reason, we propose to invert this definition.  Given a
class of oriented graphs $\mathcal{O}$, an \textit{$\mathcal{O}$-graph} is a
graph that admits an orientation that belongs to $\mathcal{O}$. In other words,
the class of $\mathcal{O}$-graphs is the family of underlying graphs of
$\mathcal{O}$. We denote this class of graphs by $\mathcal{U_O}$. For instance,
it is well-known that a graph $G$ is $2$-edge-connected if and only if it admits
a strongly connected orientation \cite{bondy2008}. So, if $\mathcal{O}$ is  the
class of strongly connected oriented graphs, then $\mathcal{U_O}$ is the class
of $2$-edge-connected graphs. 

Consider a pair of (oriented) graphs $G$ and $H$. If $G$ is homomorphic to $H$,
we will write $G\to H$, and if $G$ is an induced (oriented) subgraph of $H$, we
will write $G < H$. Given a set of oriented graphs, $F$, we denote by $Forb(F)$
the class of oriented graphs $G$ such that $H\not\to G$ for every $H\in F$. An
\textit{embedding} is a  homomorphism $\varphi\colon G\to H$ such that $G$ is
isomorphic to its image, $H[\varphi[V(G)]]$. So, $G$ embeds into $H$ if and only
if $G < H$. We extend the previously introduced notation and denote by
$Forb_e(F)$ the class of oriented graphs such that $H\not < G$ for every $H\in
F$. In particular, the class of $Forb_e(F)$-graphs is the class of graphs that
admit an $F$-free orientation. Notice that this class corresponds to the class
of $F$-graphs in the sense of Skrien~\cite{skrienJGT6}. We will often write
$Forb_e(H)$ instead of $Forb_e(\{H\})$.

Following Skrien's notation, we will use $B_1$, $B_2$, and $B_3$ to denote the
orientations of $P_3$, see Figure \ref{Fig:smallor}. Also in \cite{skrienJGT6},
Skrien proved structural characterizations of $Forb_e(F)$-graphs for every $F
\subseteq \{B_1,B_2,B_3\}$, except for $\{B_1\}$ and $\{B_2\}$; notice that
$Forb_e(B_1)$- and $Forb_e(B_2)$-graphs are actually the same class, known as
{\em perfectly-orientable graphs} (p.o.\ graphs for short). We say that a graph
is a \textit{transitive-perfectly-orientable graph} if it admits a
$\{B_1,\overrightarrow{C}_3\}$-free orientation, or equivalently, a
$\{B_2,\overrightarrow{C}_3\}$-free orientation. In other words, a
transitive-perfectly orientable graph is a graph that admits an orientation
where the out-neighbourhood of every vertex is a tournament, and every
tournament is transitively oriented. Clearly, this is a subclass of
perfectly-orientable graphs.

\begin{figure}[ht!]
\centering
\begin{tikzpicture}

\begin{scope}[yshift=3cm, scale=0.7]
\node [vertex] (1) at (-1.45,0){};
\node [vertex] (2) at (0,2){};
\node [vertex] (3) at (1.45,0){};
\node at (0,-1.2){($T_1+T_2$)};

\draw[arc]    (1)  edge  (3);
\end{scope}

\begin{scope}[xshift=-2cm, scale=0.7]
\node [vertex] (1) at (-1.45,0){};
\node [vertex] (2) at (0,2){};
\node [vertex] (3) at (1.45,0){};
\node at (0,-1.2){$\overrightarrow{C_3}$};

\draw[arc] (1) edge (2);
\draw[arc] (2) edge (3);
\draw[arc] (3) edge (1);
\end{scope}

\begin{scope}[xshift=2cm, scale=0.7]
\node [vertex] (1) at (-1.45,0){};
\node [vertex] (2) at (0,2){};
\node [vertex] (3) at (1.45,0){};
\node[] at (0,-1.2){$T_3$};

\draw [arc] (1) edge (2);
\draw [arc] (2) edge (3);
\draw [arc] (1) edge (3);
\end{scope}

\begin{scope}[xshift=-4cm, yshift=-3cm, scale=0.7]
\node [vertex] (1) at (-1.45,0){};
\node [vertex] (2) at (0,2){};
\node [vertex] (3) at (1.45,0){};
\node at (0,-1.2){$B_1$};

\draw [arc] (1) edge (2);
\draw [arc] (3) edge (2);
\end{scope}

\begin{scope}[yshift=-3cm, scale=0.7]
\node [vertex] (1) at (-1.45,0){};
\node [vertex] (2) at (0,2){};
\node [vertex] (3) at (1.45,0){};
\node at (0,-1.2){$B_2$};

\draw [arc] (2) edge (1);
\draw [arc] (2) edge (3);
\end{scope}

\begin{scope}[xshift=4cm, yshift=-3cm, scale=0.7]
\node [vertex] (1) at (-1.45,0){};
\node [vertex] (2) at (0,2){};
\node [vertex] (3) at (1.45,0){};
\node[] at (0,-1.2){$B_3$};

\draw [arc] (1) edge (2);
\draw [arc] (2) edge (3);
\end{scope}
\end{tikzpicture}
\caption{All possible orientations of non-empty graphs on three vertices.}
\label{Fig:smallor}
\end{figure}

Studying the structure of $B_1$-free orientable graphs has caught the interest
of several authors. In particular, Hartinger and Milanic, and the same authors
with Bre\v sar and Kos, have thoroughly studied this family in a series of
papers \cite{bresarDAM248, hartingerJGT2016, hartingerDM2017}. They have nice
results when the problem is restricted to certain families, e.g., they showed
that a cograph is perfectly orientable if and only if it is $K_{2,3}$-free.
Nonetheless, characterizing the class of perfectly orientable graphs through
forbidden induced subgraphs remains an open problem in the general case.

From the algorithmic point of view, Urrutia and Gavril found a polynomial time
algorithm to recognize perfectly orientable graphs \cite{urrutiaIPL41}.
Furthermore, in \cite{bangjensenJCTB59}, the authors show that for any subset
$F$ of $\{B_1,B_2,B_3\}$, there is a polynomial time algorithm to determine if a
graph admits an $F$-free orientation. They do so by reducing each of these
problems to $2$-SAT. Recall that in the classic article \cite{aspvallIPL8},
$2$-SAT is solved by proceeding on an auxiliary digraph constructed from the
$2$-SAT instance. By using these two techniques, we extend the aforementioned
result from \cite{bangjensenJCTB59} to any subset of $\{B_1, B_2, B_3, T_3\}$,
where $T_3$ is the transitive tournament of order $3$. Instead of reducing our
problem to $2$-SAT, we give an explicit construction of an auxiliary digraph
$D^+$. Then, we follow the same procedure used in \cite{aspvallIPL8} on $D^+$.
Thus, we show a certifying polynomial time algorithm to determine if a graph
belongs to $\mathcal{U}_{Forb_e(F)}$, for any set $F \subseteq \{ B_1, B_2, B_3,
T_3\}$.

In addition to the algorithm mentioned above, in this paper we extend Skrien's
work by proposing characterizations of $Forb_e(F)$-graphs when $F$ is any set of
oriented graphs on three vertices, except for $\{\overrightarrow{C_3},B_1\}$ and
$\{B_1\}$, where $\overrightarrow{C_3}$ denotes the directed $3$-cycle. Probably
the most interesting case is the family of $Forb_e(T_3)$-graphs, for which we
provide a characterization in terms of forbidden homomorphic images of a family
of graphs.   The characterization of $\mathcal{U}_{Forb_e(T_3)}$ results
surprisingly natural, and the obstructions are obtained by
``reverse-engineering'' the construction of the constraint digraph $D^+$. These
characterizations build up to our main result, which we now state.

\begin{theorem}
\label{thm:main}
  The following classes, and their intersections with complete multipartite
  graphs,  are all infinite families of $Forb_e(F)$-graphs, where $F$ is a
  set of non-empty oriented graphs on three vertices (d.u.o.\ stands for
  ``disjoint union of'').
  \begin{multicols}{2}
  \begin{enumerate}
    \item Perfectly orientable graphs.
    \item Comparability graphs. 
    \item Odd closed strip hom.-free graphs.
    \item D.u.o.\ proper circular-arc graphs. 
    \item Trivially perfect graphs. 
    \item Transitive-perfectly orientable graphs. 
    \item D.u.o.\ unicyclic graphs. 
    \item D.u.o\  triangle-free unicyclic graphs.
    \item $3$-colourable comparability graphs. 
    \item Triangle-free graphs.
    \item Clusters. 
    \item D.u.o.\ proper Helly circular-arc graphs.
    \item D.u.o.\ triangle-free proper circular-arc graphs.
    \item D.u.o.\ paths and cycles. 
    \item D.u.o.\ paths and cycles but no triangles.
    \item D.u.o.\ triangles and stars.
    \item Star forests.
    \item Stars and empty graphs.
    \item Matchings with isolated vertices.
    \item Empty graphs and $K_2$.
    \item Bipartite graphs. 
    \item Complete bipartite graphs. 
    \item Complete $3$-partite graphs. 
    \item $K_{2,3}$-free complete multipartite graphs.
    \item Complete multipartite graphs.
    \item All graphs.
  \end{enumerate}
  \end{multicols}
\end{theorem}

The rest of the paper is organized as follows. In Section \ref{sec:Algorithm},
we define the constraint digraph of a given graph $G$ and a set $F\subseteq\{
B_1, B_2, B_3, T_3\}$.   Then, we use this constraint digraph to present an
algorithm to recognize $Forb_e(F)$-graphs, when $F\subseteq\{ B_1, B_2, B_3,
T_3\}$. In Section~\ref{sec:small}, we characterize $Forb_e(F)$-graphs for most
of the cases not covered in \cite{skrienJGT6}. Section~\ref{sec:T3} is devoted
to characterize the class $\mathcal{U}_{Forb_e(T_3)}$, where we also use the
construction of the constraint digraph of Section~\ref{sec:Algorithm}. Finally,
in Section~\ref{sec:proof} we prove Theorem~\ref{thm:main}, and in
Section~\ref{sec:completemultipartite} we describe the intersections of complete
multipartite graphs with classes listed in Theorem~\ref{thm:main}. Conclusions
and some open problems are presented in Section~\ref{sec:conclusions}.

\section{Constraint Digraph}\label{sec:Algorithm}

We refer the reader to \cite{bangjensenDigraphs} and \cite{bondy2008} for
undefined basic terms. We denote the oriented graphs on three vertices as in
Figure~\ref{Fig:smallor}. Given a set $A$, we define $A \times 1 = A$ and $A
\times 0 = \varnothing$. For a statment $P$, we denote by $\mathbbm{1}_{[P]}$
the truth value of $P$. In other words, $\mathbbm{1}_{[P]}=1$ if $P$ is true,
and $\mathbbm{1}_{[P]}=0$ otherwise.

We say that any set $F \subseteq \{B_1, B_2, B_3, T_3\}$ is a {\em simple set}.
For a graph $G$ and a simple set $F$. We construct the {\em constraint digraph}
$D^+$ associated to $G$ and $F$ as follows. The vertex set, $V^+$, of $D^+$  is
the set $\{(x,y) \colon xy \in E_G\}$; notice that for every edge $xy \in E_G$,
both $(x,y)$ and $(y,x)$ belong to $V^+$. We define the following sets of arcs:
\begin{itemize}
    \item $A_1 = \{ ((y,x),(z,y)) \colon xy
      \in E_G,yz \in E_G, zx \notin E_G \}$,

    \item $A_2 = \{ ((x,y),(y,z)) \colon xy
      \in E_G,yz \in E_G, zx \notin E_G \}$,

    \item $A_3 = \{ ((x,y),(z,y)) \colon xy
      \in E_G,yz \in E_G, zx \notin E_G \}
      \cup \{ ((y,x),(y,z))$ $\colon xy \in
      E_G, yz \in E_G, zx \notin E_G \}$, and

    \item $A_t = \{ ((x,y),(y,z)) \colon xy
      \in E_G, yz \in E_G, zx \in E_G \}
      \cup \{ ((y,x),(x,z)) \colon xy \in
      E_G, yz \in E_G, zx \in E_G \}$.
\end{itemize}
Finally, we define the arc set, $A^+$, of
$D^+$ as
\[
  A^+ = (A_1 \times \mathbbm{1}_{[B_1 \in
  \mathcal{F}]}) \cup (A_2 \times \mathbbm{1}_{
  [B_2 \in \mathcal{F}]}) \cup( A_3 \times
  \mathbbm{1}_{[B_3 \in \mathcal{F}]}) \cup
  (A_t \times \mathbbm{1}_{[T_3 \in \mathcal{F}]}).
\]

In the rest of this section we will use the constraint digraph for a recognition
algorithm of certain families of $Forb_e(F)$-graphs.  We will also use $D^+$ in
Section~\ref{sec:T3} to find a structural characterization of
$\mathcal{U}_{Forb_e(T_3)}$.

We proceed to present the recognition algorithm. Given an input set $F\subseteq
\{ B_1,B_2,B_3,T_3 \}$ and a graph $G$, this algorithm finds an $F$-free
orientation of $G$, or outputs that it is not possible to find one.  We begin by
observing some properties of the constraint digraph $D^+$.

\begin{proposition}
\label{reversedarc}
  Let $G$ be a graph and $F \subseteq \{ B_1, B_2, B_3, T_3 \}$. Then, in $D^+$,
  $(x,y) \to (z,w)$ if and only if $(w,z) \to (y,x)$.
\end{proposition}

\begin{proof}
Proving one implication is enough to prove the whole statement. Observe that
$((x,y),(z,w)) \in A^+$ if and only if $((x,y),(z,w)) \in A_i$ for each $i \in
\{ 1,2,3,t \}$. We will prove the statement for the case when $((x,y),(z,w)) \in
A_1$, the other cases follow the same line of argumentation.  If
$((x,y),(z,w))\in A_1$ then $w = x$, $yx \in E_G$, $xz \in E_G$ and $zy \notin
E_G$. Thus $zx \in E_G$, $xy \in E_G$ and $yz \notin E_G$, therefore
$((x,z),(y,x)) \in A_1$. Hence, $((w,z), (y,x))\in A_1$ if and only if
$((x,y),(z,w)) \in A_1$.
\end{proof}

From here, the following two propositions are easy to obtain.

\begin{proposition}
\label{reversedpath}
  Let $G$ be a graph and $F \subseteq \{ B_1,B_2, B_3,T_3\}$. There is a
  directed path from $(x,y)$ to $(z,w)$ in $D^+$ if and only if there is a
  directed path  from $(w,z)$ to $(y,x)$ in $D^+$.
\end{proposition}

\begin{proof}
Proceed by induction over the length of the directed path. Notice that
Proposition~\ref{reversedarc} is the base case. Use again
Proposition~\ref{reversedarc} in the inductive step.
\end{proof}

Let $D$ be a digraph and let $\overleftarrow{D}$ be the digraph obtained from
$D$ by reversing every arc. A digraph $D$ is {\em skew-symmetric} if it is
isomorphic to $\overleftarrow{D}$.

\begin{proposition}
\label{skewsymmetric}
  Let $G$ be a graph and $F\subseteq\{B_1,B_2,B_3,T_3\}$. The constraint digraph
  of $G$ and $F$ is skew-symmetric.
\end{proposition}

\begin{proof}
Let $D$ be a digraph. Let $D^+$ be the constraint digraph of $G$ and $F$.
Consider the function $\varphi \colon V^+ \to V^+$ defined by $\varphi((x,y)) =
(y,x)$. By Proposition~\ref{reversedarc}, we conclude that $\varphi$ is a
digraph isomorphism between $D^+$ and $\overleftarrow{D^+}$.
\end{proof}

By the isomorphism shown in the previous proof, every strong component $S$ in
$D^+$ has a dual component, $\overline{S}$ (which might be equal to $S$),
induced  by the vertices of the form $(y,x)$ where $(x,y)\in S$. By
Proposition~\ref{reversedpath}, a strong component $S_1$ reaches another one
$S_2$, if and only if $\overline{S_2}$ reaches $\overline{S_1}$. A well-known
algorithm of Tarjan \cite{tarjanSIAMJC1972} generates the strong components of a
digraph in reverse topological order (i.e.\ if $S_1$ reaches $S_2$ then $S_2$ is
generated before $S_1$).

Let us go back to the construction of the constraint digraph. Suppose that we
want to find an $F$-free orientation of $G$. An arc $((x,y),(z,w))$ in $D^+$
tells us that, in order to achieve such an orientation, if we orient the edge
$xy$ from $x$ to $y$, then we must orient the edge $zw$ from $z$ to $w$.
Inductively, if there is a path from $(x,y)$ to $(z,w)$ and we orient the edge
$xy$ from $x$ to $y$ then we must orient the edge $zw$ from $z$ to $w$. Thus, if
$(x,y)$ and $(y,x)$ belong to the same strong component, $G$ does not admit an
$F$-free orientation. In fact the reverse implication is also true.  To see
this, we consider the famous $2$-SAT Algorithm \index{Algorithm!2-SAT} due to
Tarjan \cite{aspvallIPL8}, and use it on the constraint digraph $D^+$ associated
to a graph $G$ and a set $F\subseteq \{B_1,B_2,B_3,TT_3\}$. This procedure is
described in Algorithm \ref{alg:master}.

\begin{algorithm}[ht!]
  \SetAlgorithmName{Algorithm}{}
    \DontPrintSemicolon
    \SetKwData{False}{false}\SetKwData{True}{true}
    \SetKwFunction{New}{new}\SetKwFunction{End}{end}\SetKwFunction{Used}{used}
    \SetKwInOut{Input}{input}\SetKwInOut{Output}{output}
  
    \KwIn{A graph $G$ and a set $F\subseteq \{B_1,B_2,B_3,TT_3\}$.}
    \KwOut{A $\{\texttt{true},\texttt{false} \}$-colouring of
    the vertices in $D^+$, or a strong component $S$ such that $S =
    \overline{S}$.}
    \BlankLine
    Construct the constraint digraph $D^+$ associated to $G$ and $F$\;
    Generate the strong components of $D^+$ in reverse topological order\;
    {\For{each strong component $S$ of $D^+$}{
     \uIf{$S = \overline{S}$}{
      \Return $S$\;
    }\uElseIf{vertices in $S$ are not marked}{
      mark each vertex in $S$ $\texttt{true}$ and each vertex in $\overline{S}$
      $\texttt{false}$\;
    }
    }
    }
    {\Return the $\{\texttt{true},\texttt{false} \}$-colouring of the vertices
    of $D^+$}
    \caption{Recognition of $F$-free orientable graphs} \label{alg:master}
  \DecMargin{1em}
  \end{algorithm}

Clearly, Algorithm~\ref{alg:master} stops by determining that there exists a
strong component $S$ such that $S = \overline{S}$ only if there is a vertex
$(x,y)\in V^+$ in the same strong component as $(y,x)$. Otherwise a
$\{\texttt{true,false}\}$-colouring of $D^+$ is obtained, which induces an
$F$-free orientation of $G$. We prove this fact in the following proposition.

\begin{proposition}
\label{prop:algorithm1}
  Let $G$ be a graph and $F\subseteq \{B_1, B_2, B_3, T_3\}$. If Algorithm
  \ref{alg:master} outputs a $\{\texttt{true},\texttt{false} \}$-colouring of
  the vertices in $D^+$, then vertices with colour \texttt{true} induce an
  $F$-free orientation of $G$.
\end{proposition}

\begin{proof}
Clearly, if $(x,y)$ is marked with \texttt{true}, then $(y,x)$ is marked with
\texttt{false}. Also, every vertex receives one and only one truth colour. Hence
the \texttt{true}-coloured vertices of $D^+$ induce an orientation of $G$; that
is, if $(x,y)$ is marked \texttt{true}, then $xy$ is oriented as $(x,y)$. We now
prove that this orientation is an $F$-free orientation of $G$. To do so, we must
prove that for any two oriented edges $(x,y),(w,z) \in V^+$ that induce an
oriented graph in $F$, then at least one is marked with \texttt{false}. By
construction of $A^+$, it must happen that if $(x,y)$ and $(w,z)$ induce an
oriented graph in $F$ then $(x,y) \to (z,w)$ and $(w,z) \to (y,x)$. Hence, we
show that if $(x,y)$ is marked with \texttt{true} and $(x,y) \to (z,w)$, then
$(z,w)$ is also marked with \texttt{true}. Since the algorithm marks all the
vertices in the same strong component at once, it suffices to show that for any
two strong components $S_1$ and $S_2$ of $D^+$, if $S_1$ is
\texttt{true}-coloured and $S_1$ reaches $S_2$, then $S_2$ is also
\texttt{true}-coloured. Suppose that $S_1$ is marked with \texttt{true} and it
reaches $S_2$, but $S_2$ is \texttt{false}-coloured. Since $S_1$ reaches $S_2$,
then $S_2<S_1$, where $<$ is the reverse topological order of the strong
components of $D^+$. Since $S_2$ is marked with \texttt{false} it means that
$\overline{S_2}$ was processed before $S_2$ (i.e.\ $\overline{S_2} < S_2$).
Analogously, we see that $S_1 < \overline{S_1}$. The transitivity of $<$ implies
that $\overline{S_2} < \overline{S_1}$. Since $S_1$ reaches $S_2$, by
Proposition~\ref{reversedpath}, $\overline{S_2}$ reaches $\overline{S_1}$, then
$\overline{S_1} < \overline{S_2}$. The previous inequalities yield the following
chain, $\overline{S_1} < \overline{S_2} < S_2 < S_1 < \overline{S_1}$. From
which we conclude that $\overline{S_1} = \overline{S_2}$; equivalently $S_1 =
S_2$. This contradicts that the algorithm does not assign two different truth
values to the same component. Therefore if $S_1$ reaches $S_2$ and $S_1$ is
marked with \texttt{true}, $S_2$ is marked with \texttt{true} as well.
\end{proof}

These results build up to the following one.

\begin{theorem}
\label{thm:algorithm1}
  Let $G$ be a graph and $F\subseteq\{B_1, B_2, B_3, T_3\}$. The following are
  equivalent:
\begin{itemize}
    \item $G$ admits an $F$-free orientation.

    \item There are no vertices $(x,y),(y,x)\in V^+$ contained in the same
      strong connected component of $D^+$.

    \item For any strong component $S$, $S \cap \overline{S} = \varnothing$
      (i.e.\ $S \ne \overline{S}$).
\end{itemize}
\end{theorem}

\begin{proof}
The equivalence between the second and third item is trivial. On the paragraph
preceding Algorithm \ref{alg:master} it was shown that the second statement
implies the first one. The remaining implication is proved by Algorithm
\ref{alg:master} and Proposition \ref{prop:algorithm1}.
\end{proof}

The order of $D^+$ is $2m$, where $m$ is the number of edges of $G$. Also note
that $d_{D^+}((x,y))\leq d_G(x)+d_G(y)$ so, $|A^+|\leq m\Delta(G)\leq mn$. Since
the general step of Algorithm~\ref{alg:master} runs in $O(|V^+|)$ time and
Tarjan's Algorithm \cite{tarjanSIAMJC1972} runs in $O(|V^+| + |A^+|)$ time, our
algorithm runs in $O(mn)$ once $D^+$ is constructed --- to construct $D^+$ we
must process all sets of $3$ vertices, which takes cubic time in $|V|$. These
arguments together with Theorem~\ref{thm:algorithm1} show the the following
statement holds.

\begin{corollary}\label{cor:somesets}
Given an input graph $G$ and an input set $F\subseteq\{B_1,B_2,B_3,T_3\}$,
it is in $P$ to test if $G$ admits an $F$-free orientation. In particular, 
Algorithm~\ref{alg:master} decides if $G$ admits an $F$-free orientation
in cubic time. 
\end{corollary}

\section{Graph properties and small forbidden orientations}
\label{sec:small}

In this section, we study families of $Forb_e(F)$-graphs when $F$ consists of
oriented graphs on three vertices. In \cite{skrienJGT6}, Skrien studied the
cases when $F$ is a set of orientations of $P_3$. For this reason, we study
$Forb_e(F)$-graphs when either $T_1+ T_2 \in F$ or $F$ contains at least one
orientation of $C_3$. We begin by observing the following simple lemma. 

\begin{lemma}
\label{lem:intersection}
  Let $F$ be a set of oriented graphs and consider any graph $H$. If $F_H$ is
  the set of all orientations of $H$, then the class of $Forb_e(F\cup
  F_H)$-graphs is the intersection of $\mathcal{U}_{Forb_e(F-F_H)}$ and $H$-free
  graphs.
\end{lemma}

\begin{proof}
If $G$ is an $H$-free graph, and $G'$ is an $(F-F_H)$-free orientation of $G$,
then $G'$ is an $(F\cup F_H)$-free orientation of $G$. On the other hand, if
$G\in \mathcal{U}_{Forb_e(F\cup F_H)}$ and $G'$ is an $(F\cup F_H)$-free
orientation of $G$, then $G'$ is $(F- F_H)$-free and $F_H$-free. So, $G \in
\mathcal{U}_{Forb_e(F-F_H)}$ and $G$ is $H$-free.
\end{proof}

In particular, since $T_1+ T_2$ is the unique orientation of $K_1+K_2$, then the
class of $Forb_e(F \cup \{T_1 + T_2 \})$-graphs is the intersection of
$\mathcal{U}_{Forb_e( F - \{T_1 + T_2\})}$ and complete multipartite graphs.
Similarly, the class of $Forb_e(F \cup \{\overrightarrow{C}_3,T_3\})$-graphs is
the intersection of $\mathcal{U}_{Forb_e(F-\{\overrightarrow{C}_3,T_3\})}$ and
triangle-free graphs.

\begin{proposition}
\label{prop:somesets}
  Let $F$ be a set of oriented graphs on $3$ vertices. If $F\subseteq \{B_1,
  B_2, B_3, T_3, T_1+T_2\}$ or $\{\overrightarrow{C}_3,T_3\}\subseteq F$, then
  it is in $P$ to test if an input graph admits an  $F$-free orientation.
\end{proposition}

\begin{proof}
This statement follows from Corollary~\ref{cor:somesets} and
Lemma~\ref{lem:intersection}.
\end{proof}

It is direct to verify that if the set of forbidden oriented graphs consists of
connected graphs, then the associated hereditary property is closed under
disjoint unions. Thus, it suffices to study connected graphs.

\begin{table}[ht!]
\begin{center}
  \begin{tabular}{| c | l |}
    \hline
    Forbidden oriented graphs & Graph family\\
    \hline
    $B_1,B_2,B_3$ & Complete graphs. \\
    \hline
    $B_1,B_2$ &  Proper circular-arc graphs.\\
    \hline
    $B_1,B_3$ &  Nested interval graphs.\\
    \hline
    $B_2,B_3$ &  Nested interval graphs. \\
    \hline
    $B_1$ &  Open \\
    \hline
    $B_2$ &  Open \\
    \hline
    $B_3$ & Comparability graphs. \\
    \hline
  \end{tabular}
  \caption{On the left we see a set of forbidden oriented graphs, and on the
    right, the family it characterizes.  This table is taken from
    \cite{skrienJGT6}.}
\end{center}
\end{table}\label{Tab:skrien}

Skrien's results from \cite{skrienJGT6} are included in Table~\ref{Tab:skrien}.
Recall that he found an alternative characterization for all sets containing
orientations of $P_3$, except for perfectly orientable graphs. Bang-Jensen,
Huang and Prisner also studied p.o.\ graphs, in particular, they proved the
following result in \cite{bangjensenJCTB59}.

\begin{proposition}\cite{bangjensenJCTB59}
\label{bangjensen}
  Every graph with exactly one induced cycle of length greater than $3$ is
  perfectly orientable.
\end{proposition}

This result can be equivalently restated as follows: Every triangle-free graph
is perfectly orientable if it has only one induced cycle. With a simpler proof
than the one found in \cite{bangjensenJCTB59}, we prove the biconditional
version of this result, which is a corollary to the following proposition.

\begin{proposition}
\label{B1T3-free}
  The following statements are equivalent for a connected graph $G$:
  \begin{enumerate}
    \item $G$ admits a $\{B_1,T_3\}$-free orientation,

    \item $G$ admits an orientation such that $d^+(x) \le 1$ for every vertex $x
      \in V_G$,

    \item there is function $f \colon V_G \to V_G$ such that $E_G = \{xy \colon
      x \ne y, f(x) = y \}$,

    \item $G$ is unicyclic,

    \item $G$ has no more edges than vertices.
  \end{enumerate}
\end{proposition}

\begin{proof}
It is not hard to notice that the first two items are equivalent, and so are the
second and third one. It is also straightforward to show that if $G$ has no more
edges than vertices, then $G$ is unicyclic (recall that $G$ is connected), so
$4$ is an implication of $5$.  Now we prove that the second item implies the
fifth one. Let $D_G$ be an orientation of $G$ such that $d^+(x)\le 1$ for every
vertex $x$ of $G$.  Consider the function $i \colon A_{D_G}\to V_G$ where
$i((x,y)) = x$. Since $d^+(x)\le 1$, $i$ is an injective function. Thus $|E_G| =
|A_{D_G}| \le |V_G|$. To conclude the proof we show that if $G$ is unicyclic, it
admits an $\{B_1,T_3\}$-free orientation. If $G$ is a tree, root $G$ in any
vertex and orient the edges from descendent to ancestor. If $G$ is a cycle,
orient $G$ in a cyclic way. In any other case, let $C$ by the only cycle in $G$.
Orient $C$ in a cyclic way. Notice that $G/C$ is a tree. Root $G/C$ in the
vertex corresponding to $C$. Orient the edges in $G/C$ from descendent to
ancestor. We have oriented all edges in $G$ now, and it it not hard to notice
that this orientation is $\{B_1,T_3\}$-free.
\end{proof}

\begin{corollary}
\label{B1C3T3-free}
  A graph $G$ admits a $\{B_1, \overrightarrow{C_3}, T_3\}$-free orientation if
  and only if $G$ is unicyclic and triangle free.
\end{corollary}

\begin{proof}
  Suppose $G$ admits a $\{B_1,\overrightarrow{C_3}, T_3\}$-free orientation.
  Clearly, $G$ is triangle-free and by Proposition~\ref{B1T3-free}, $G$ is also
  a unicyclic graph. On the other hand, consider a triangle-free unicyclic graph
  $G$. By Proposition~\ref{B1T3-free}, it admits a $\{B_1, T_3\}$-free
  orientation $D_G$. Since $G$ is triangle-free, $D_G$ is $\{B_1,
  \overrightarrow{C_3}, T_3\}$-free.
\end{proof}

Another subclass of perfectly orientable graphs is the class of graphs that
admit a $\{B_1,T_1+T_2\}$-free orientation.

\begin{proposition}
\label{B1T1+T2-free}
  The following statements are equivalent for a graph $G$:
  \begin{enumerate}
    \item $G$ admits a $\{B_1, T_1+T_2\}$-free orientation.
    \item $G$ admits a $\{B_2, T_1+T_2\}$-free orientation.
    \item $G$ is a $K_{2,3}$-free complete multipartite graph.
  \end{enumerate}
\end{proposition}

\begin{proof}
  The equivalence between the first two items is obvious. By
  Lemma~\ref{lem:intersection}, a graph $G$ admits a $\{B_1, T_1+T_2\}$-free
  orientation if and only if it is a perfectly-orientable complete multipartite
  graph. In \cite{hartingerJGT2016}, the authors showed that a cograph is
  perfectly orientable graph if and only if it is $K_{2,3}$-free. The claim now
  follows since complete multipartite graphs are cographs.
\end{proof}

By further restricting the family described in Proposition~\ref{B1T1+T2-free},
we immediately obtain the following simple corollary.

\begin{corollary}
\label{B1T1+T2T3C3-free}
  The following statements are equivalent for a graph $G$:
  \begin{enumerate}
    \item $G$ admits a $\{B_1, T_1+T_2, T_3,\overrightarrow{C}_3\}$-free
      orientation.
    \item $G$ admits a $\{B_2, T_1+T_2, T_3,\overrightarrow{C}_3\}$-free
      orientation.
    \item $G$ is a $K_{2,3}$-free complete bipartite graph.
    \item $G$ is a star or $G = C_4$.
  \end{enumerate}
\end{corollary}

When $F = \{B_3, T_3, \overrightarrow{C_3}\}$, the class
$\mathcal{U}_{Forb_e(F)}$ has a well-known characterization, which is a
particular case of the Gallai-Hasse-Roy-Vitaver Theorem.

\begin{proposition}
\label{B3C3T3-free}
  A graph is bipartite if and only if it admits an $\{B_3, T_3,
  \overrightarrow{C_3}\}$-free orientation.
\end{proposition}

In \cite{skrienJGT6}, Skrien shows that a graph is a proper circular arc graph
if and only if it admits a $\{B_1,B_2\}$-free orientaion.  A proper circular-arc
graph is a graph that admits an intersection model where no arc is contained in
another. A family of sets $\mathcal{A}$ is said to have the Helly property, if
for any subfamily $\mathcal{B} \subseteq \mathcal{A}$ such that for any two sets
$A,B \in \mathcal{B}$, $A \cap B \ne \varnothing$, then the intersection of all
sets in $\mathcal{B}$ is non-empty. A (proper) Helly cicular-arc graph is a
graph that admits an intersection model that satisfies the Helly property (and
no arc is contained in another). We extend Skrien's result to proper Helly
circular-arc graphs.

\begin{proposition}
\label{B1B2C3-free}
  A graph $G$ admits a $\{B_1, B_2, \overrightarrow{C_3}\}$-free orientation if
  and only if $G$ is a proper Helly circular-arc graph.
\end{proposition}

\begin{proof}
  Let $G$ be a graph that admits a $\{B_1,B_2, \overrightarrow{C_3}\}$-free
  orientation. By line two of Table~\ref{Tab:skrien}, we know that $G$ must be a
  proper circular-arc graph. Corollary 5 in \cite{linDAM2013}, shows that a
  proper circular-arc graph is a proper Helly circular-arc graph if it contains
  neither the Hajos graph nor a $4$-wheel as an induced subgraph. It is not hard
  to notice that neither of those graphs admit a $\{B_1,
  B_2,\overrightarrow{C_3}\}$-free orientation. Thus, since $G$ is a proper
  circular-arc graph, $G$ must be a proper Helly circular-arc graph.

  In \cite{mckeeDM2003}, it is proved that a model of a proper circular-arc
  graph is the model of a proper Helly circular-arc graph if and only if no two
  nor three arcs cover its circle. Consider a proper Helly circular-arc graph
  $G$. Let $\mathcal{A} = \{A_1, A_2, \dots, A_n\}$ be a model of $G$ where no
  three arcs cover the circle. Moreover, we can assume that no end points of the
  arcs in $\mathcal{A}$ coincide. Let us denote by $l_i$ the anti-clockwise end
  point of $A_i$, and by $r_i$ the clockwise end point. We denote by $D_G$ the
  following orientation of $G$. Consider an edge $A_iA_j\in E_G$. By moving in a
  clockwise motion around the circle, we see the endpoints of $A_i$ and $A_j$
  form the sequence $[l_i, l_j,r_i,r_j]$ or $[l_j,l_i,r_j,r_i]$. We orient
  $A_iA_j$ form $A_i$ to $A_j$ when we see $[l_i,l_j, r_i,r_j]$, in the other
  case we orient it from $A_j$ to $A_i$. Bearing in mind that there are no three
  arcs that cover the circle, it is easy to see $D_G$ is
  $\{B_1,B_2,\overrightarrow{C_3}\}$-free.
\end{proof}

Interval graphs are particular instances of circular-arc graphs. In
\cite{skrienJGT6}, Skrien showed that \textit{nested interval} graphs correspond
to $Forb_e(\{B_1,B_3\})$-graphs. A \textit{nested interval} graph is a graph $G$
that admits an intersection model that consists of nested intervals of the real
line. Equivalently, a graph $G$ is a nested interval graph if and only if it is
$\{P_4,C_4\}$-free \cite{golumbicDM24}. These graphs are also called
trivially-perfect graphs.

\begin{proposition}
\label{B1B3C3-free}
  For a graph $G$ the following statements are equivalent:
  \begin{enumerate}
    \item $G$ admits a $\{B_1,B_3\}$-free orientation.
    \item $G$ admits a $\{B_1,B_3, \overrightarrow{C}_3\}$-free orientation.
    \item $G$ is a nested interval graph.
    \item $G$ is a trivially perfect graph.
  \end{enumerate}
\end{proposition}

\begin{proof}
The equivalent between the first and third statement is proved in
\cite{skrienJGT6}. The equivalence between the last two items is argued above
this proposition. Clearly, the second statement is a particular case of the
first one. To conclude the proof we show that the third statement implies the
second one. This is immediate considering the orientation defined by the
inclusion of the nested intervals. 
\end{proof}

Since every graph admits an acyclic orientation, every graph admits a
$\overrightarrow{C_3}$-free orientation.  On the contrary, not every graph
admits a $T_3$-free orientation. Recall that a graph is {\em locally bipartite}
if the open neighbourhood of every vertex induces a bipartite graph.

\begin{proposition}
\label{obs1:T3}
For any graph $G$ the following statements hold:
\begin{itemize}
     \item if $G$ is $3$-colourable, then it admits a $T_3$-free orientation,

     \item if $G$ admits a $T_3$-free orientation, then it is $K_4$-free,

     \item if $G$ admits a $T_3$-free orientation, then it is locally bipartite.
\end{itemize}
\end{proposition}

\begin{proof}
  Let $G$ be graph with a proper colouring $(V_0, V_1, V_2)$ . By orienting the
  edges of $G$ from $V_i$ to $V_{i+1}$, with subindices taken modulo 3, we
  obtain a $T_3$-free orientation of $G$. In order to prove the second item, it
  suffices to notice that $K_4$ does not admit a $T_3$-free orientation. Let
  $D_G$ be a $T_3$-free orientation of a graph $G$. For any vertex $x \in V_G$,
  the sets $N_{D_G}^+(x)$ and $N_{D_G}^-(x)$ are a partition of $N_G(x)$. Since
  $D_G$ is $T_3$-free, $N_{D_G}^+(x)$ and $N_{D_G}^-(x)$ are independent sets.
\end{proof}

As we will see later, the statements in the previous proposition are far from
being necessary and sufficient conditions for a graph $G$ to admit a $T_3$-free
orientation. For the moment, recall the well known result of Mycielski stating
that the chromatic number on triangle-free graphs is unbounded
\cite{mycielskiCM1995}. Thus, there are graphs with arbitrary large chromatic
number that admit a $T_3$-free orientation.   Nonetheless, for perfect graph,
the first condition of the previous proposition actually characterizes graphs
admitting a $T_3$-free orientation.

\begin{proposition}
\label{perfectgraphsT3}
  A perfect graph $G$ admits a $T_3$-free orientation if and only if it is
  $3$-colourable.
\end{proposition}

\begin{proof}
  Consider a perfect graph $G$. By Proposition~\ref{obs1:T3}, if $G$ is
  $3$-colourable it admits a $T_3$-free orientation. On the other hand, suppose
  that $G$ admits a $T_3$-free orientation. By Proposition~\ref{obs1:T3}, $G$ is
  $K_4$-free. Since $G$ is perfect, $G$ is $3$-colourable.
\end{proof}

\begin{corollary}
\label{K1K2T3-free}
  A graph admits a $\{T_1+ T_2,T_3\}$-free orientation if and only if it is a
  complete $3$-partite graph.
\end{corollary}

\begin{proof}
  By part 1 of Lemma~\ref{lem:intersection}, the class $\mathcal{U}_{Forb_e(T_1+
  T_2,T_3)}$ is the intersection of $\mathcal{U}_{Forb_e(T_3)}$ and complete
  multipartite graphs. Since complete multipartite graphs are perfect graphs,
  the claim follows by Proposition~\ref{perfectgraphsT3}.
\end{proof}

Since comparability graphs are perfect graphs, the following propositions stem
from Proposition~\ref{perfectgraphsT3}.

\begin{proposition}
\label{B3T3-free}
  A graph admits a $\{B_3,T_3\}$-free orientation if and only if it is a
  $3$-colourable comparability graph.
\end{proposition}

\begin{proof}
  If a graph $G$ admits a $\{B_3,T_3\}$-free orientation, then it is a
  comparability graph. Thus, $G$ is a perfect graph that admits a $T_3$-free
  orientation. By Proposition~\ref{perfectgraphsT3}, $G$ is a $3$-colourable
  comparability graph. Now suppose that $G$ is a $3$-colourable comparability
  graph. Since $G$ is perfect, it is $K_4$-free. Consider the partial order of
  the vertices, $<$, induced by the edges of $G$. Let $X_1 = \{ x \in V_G \colon
  x$ is $<$-minimal$\}$, $X_3 = \{ x \in V_G \colon x$ is $<$-maximal$\}$ and
  $X_2 = V_G - (X_1 \cup X_3)$. It follows from the construction of $X_i$, $i
  \in \{ 1,2,3 \}$, and the fact that $G$ is $K_4$-free, that the sets $X_i$ is
  an independent set for $i \in \{ 1,2,3 \}$. Orient the edges from $X_1$ to
  $X_2$, from $X_2$ to $X_3$ and from $X_3$ to $X_1$; name this orientation
  $D_G$. Clearly, $D_G$ is $T_3$-free. In order to show that $D_G$ is also
  $B_3$-free, consider three vertices $x,y,z \in V_G$, that induce a path on
  $G$. Since $\{ x,y,z \}$ does not induce a triangle, it may not happen that $x
  < y < z$. Thus $x<y$ and $z<y$, or $y<x$ and $y<z$. Then $\{ x,y,z \}$ induces
  either a $B_1$ or $B_2$ in $D_G$. Concluding that $D_G$ is a $\{ B_3,T_3
  \}$-free orientation of $G$.
\end{proof}

Before proceeding to study the non perfect graphs that admit a $T_3$-free
orientation, allow us to study three very simple subclasses.

\begin{proposition}
\label{B1B2T3-free}
  A graph $G$ admits a $\{B_1,B_2,T_3\}$-free orientation if and only if
  $\Delta(G) \le 2$. Equivalently, $G$ admits a $\{ B_1,B_2,T_3 \}$-free
  orientation if and only $G$ is a dijsoint union of paths and cycles.
\end{proposition}

\begin{proof}
  Recall that $\Delta(G) \le 2$ if and only if $G$ is a disjoint union of paths
  and cycles. Suppose that there is a vertex $x \in V_G$ with at least three
  distinct neighbours, $y,z,w$. Let $D_G$ be an orientation of $G$. Without loss
  of generality, $y$ and $z$ will be in-neighbours of $x$ in $D_G$. If $yz \in
  E_G$ then $\{ x,y,z \}$ will induce a $T_3$ in $D_G$. On the other hand, if
  $yz \not \in E_G$, $\{ x,y,z \}$ will induce a $B_1$ in $D_G$. Thus if
  $\Delta(G) \ge 3$, $G$ does not admit a $\{B_1, B_2, T_3\}$-free orientation.
  To conclude the proof, consider a disjoint union of paths and cycles $G$. By
  orienting every cycle and path of $G$ in a directed way, we obtain a
  $\{B_1,B_2,T_3\}$-free orientation of $G$.
\end{proof}

\begin{proposition}
\label{B1B3T3-free}
  A connected graph $G$ admits a $\{B_1,B_3,T_3\}$-free orientation if and only
  if $G$ is a star or a triangle.
\end{proposition}

\begin{proof}
  It is trivial to find a $\{B_1,B_3,T_3\}$-free orientation of a star or a
  triangle. Recall that a connected graph $G$ is a star if and only if $G$ is
  $\{P_4,C_4,C_3\}$-free. Notice that neither $P_4$ nor $C_4$ admit a
  $\{B_1,B_3\}$-free orientation. Thus if $G$ does not contain a triangle and
  admits a $\{B_1,B_3,T_3\}$-free orientation, $G$ is a star. On the contrary,
  if $G$ contains a triangle, observe that neither of the three connected
  supergraphs of $C_3$ on four vertices, admit a $\{B_1,B_3,T_3\}$-free
  orientation. Thus, if $G$ contains a triangle $C$, then $G=C$.
\end{proof}

\begin{corollary}
\label{B1B3C3T3-free}
  A graph $G$ admits a $\{B_1,B_3,T_3,\overrightarrow{C}_3\}$-orientation if and
  only if it is a star forest. Equivalently, $G$ admits a
  $\{B_2,B_3,T_3,\overrightarrow{C}_3\}$-orientation if and only if it is a star
  forest.
\end{corollary}

\section{$Forb_e(T_3)$-graphs}
\label{sec:T3}

The following results build up to characterize the family of graphs that admit a
$\{T_3\}$-free orientation.

\begin{proposition}
\label{homduality}
  Consider a set of tournaments $F$ and a $Forb_e(F)$-graph $H$. If a graph $G$
  admits a homomorphism $\varphi \colon G \to H$, then $G$ admits an $F$-free
  orientation.
\end{proposition}

\begin{proof}
  Consider an $F$-free orientation $D_H$ of $H$. We obtain an orientation $D_G$
  of $G$ in the following way, there is an arc $(x,y)$ in $D_G$ if and only if
  $(\varphi(x), \varphi(y))$ is an arc in $D_H$. Since $\varphi$ is a graph
  homomorphism, by the way we chose to orient the edges of $G$, $\varphi$
  induces a digraph homomorphism $\varphi_D \colon D_G \to D_H$. Thus, every
  tournament $T$ in $D_G$, can be embedded in $D_H$. Since $F$ consists of
  tournaments and $D_H$ is an $F$-free orientation of $H$, $D_G$ is also an
  $F$-free orientation of $G$.
\end{proof}

Recall that, if a graph $G$ admits a homomorphism to another graph $H$, we write
$G \to H$, and $G \not \to H$ otherwise. If $\mathcal{F}$ is a set of graphs, we
write $\mathcal{F} \not \to H$, if $G \not \to H$ for every graph $G \in
\mathcal{F}$.

\begin{corollary}
\label{cor:motivation}
  For every set of tournaments $F$, there is a set of graphs $\mathcal{F}$ such
  that for any graph $G$, $G$ admits an $F$-free orientation if and only if
  $\mathcal{F} \not \to G$.
\end{corollary}

\begin{proof}
  By Proposition~\ref{homduality} an example of such a set, is the set of graphs
  that do not admit an $F$-free orientation.
\end{proof}

This corollary motivates the characterization we propose of
$Forb_e(T_3)$-graphs: We describe a set of graphs $\mathcal{F}$ such that a
graph $G$ admits a $T_3$-free orientation if and only if $\mathcal{F} \not\to
G$.

We begin by introducing some definitions. Consider two parallel paths on the
plane $P$ and $Q$. A \textit{strip} is a graph $G$ obtained from $P$ and $Q$ as
follows. First, add one edge joining the initial vertices of $P$ and $Q$, and
one joining the end vertices of $P$ and $Q$. Then, triangulate the region
between $P$ and $Q$ in such a way that every new edge is incident with one
vertex in $P$ and one in $Q$.  We call $P$ and $Q$ the \textit{bounding paths}
of $G$. In the top of Figure~\ref{fig:strips}, we illustrate an example of a
strip. 

A \textit{closed strip} is obtained from a strip $G$ with bounding paths $P$ and
$Q$ by identifying the first and final vertices of $P$, and the first and final
vertices of $Q$. Similarly,  a \textit{M\"obius strip} is obtained by identify
the first vertex of $P$ with the final vertex of $Q$, and the first vertex of
$Q$ with the final vertex of $P$. We will abuse nomenclature and call $P$ and
$Q$ the bounding paths of the closed (resp.\ M\"obius) strip --- notice that the
quotients of $P$ and $Q$ are cycles in the corresponding closed strip. In
Figure~\ref{fig:strips}, we depict an example of a closed strip and of a
M\"obius strip.  
\begin{figure}[ht!]
\label{fig:strips}
\begin{center}
\begin{tikzpicture}
 
\begin{scope}[xshift=0cm, scale=0.5]
\node[]  at (-4,0){$P$};
\node[]  at (-4,2){$Q$};

\node [vertex, label = -90:\footnotesize{$p_1$}] (1) at (-3,0){};
\node [vertex, label = 90:\footnotesize{$q_1$}] (2) at (-3,2){};
\node [vertex, label = 90:\footnotesize{$q_2$}] (3) at (-1,2){};
\node [vertex, label = -90:\footnotesize{$p_2$}] (4) at (-1,0){};
\node [vertex, label = -90:\footnotesize{$p_3$}] (5) at (1,0){};
\node [vertex, label = 90:\footnotesize{$q_3$}] (6) at (1,2){};
\node [vertex, label = -90:\footnotesize{$p_4$}] (7) at (3,0){};
\node [vertex, label = 90:\footnotesize{$q_4$}] (8) at (3,2){};
\node [vertex, label = 90:\footnotesize{$q_5$}] (9) at (5,2){};

\foreach \from/\to in {1/2,2/3,3/1,1/4,3/4,4/5,3/5,3/6,5/6,6/7,5/7, 7/8, 7/9, 6/8, 8/9}
\draw [edge] (\from) to (\to);
\end{scope}

\begin{scope}[xshift=-3cm, yshift = -3cm, scale=0.5]

\node [vertex, label = 90:{\footnotesize{$p_1, p_4$}}] (1) at (90:1){};
\node [vertex, label= 0:{\footnotesize{$p_2$}} ] (4) at (-30:1){};
\node [vertex, label= 270:{\footnotesize{$p_3$}}] (5) at (210:1){};
\node (7) at (90:1){};
\foreach \from/\to in {1/4, 4/5, 5/7}
\draw [edge] (\from) to (\to);

\node [vertex, label = 90:{\footnotesize{$q_1, q_5$}}] (2) at (45:3.3){};
\node [vertex, label = -90:{\footnotesize{$q_2$}}] (3) at (-45:3.5){};
\node [vertex, label = -90:{\footnotesize{$q_3$}}] (6) at (225:3.5){};
\node [vertex, label = 90:{\footnotesize{$q_4$}}] (8) at (135:3.5){};
\node (9) at (45:3.5){};
\foreach \from/\to in {2/3, 3/6, 6/8, 8/9}
\draw [edge] (\from) to (\to);

\foreach \from/\to in {1/2, 5/6, 7/8, 3/5, 3/4}
\draw [edge] (\from) to (\to);

\draw [edge] (3) to [bend right] (1);
\draw [edge] (6) to [bend left] (1);

\end{scope}

\begin{scope}[xshift=3cm, yshift = -3cm, scale=0.5]

\node [vertex, label = 90:{\footnotesize{$q_1, p_4$}}] (1) at (90:1){};
\node [vertex, label= 0:{\footnotesize{$p_2$}} ] (4) at (-30:1){};
\node [vertex, label= 270:{\footnotesize{$p_3$}}] (5) at (210:1){};
\node (7) at (90:1){};

\node [vertex, label = 90:{\footnotesize{$p_1, q_5$}}] (2) at (45:3.3){};
\node [vertex, label = -90:{\footnotesize{$q_2$}}] (3) at (-45:3.5){};
\node [vertex, label = -90:{\footnotesize{$q_3$}}] (6) at (225:3.5){};
\node [vertex, label = 90:{\footnotesize{$q_4$}}] (8) at (135:3.5){};
\node (9) at (45:3.5){};

\foreach \from/\to in {2/4, 4/5, 5/7}
\draw [edge] (\from) to (\to);

\foreach \from/\to in {2/3, 3/6, 6/8, 8/9}
\draw [edge] (\from) to (\to);

\foreach \from/\to in {1/2, 5/6, 7/8, 3/5, 3/4}
\draw [edge] (\from) to (\to);

\draw [edge] (3) to [bend right] (1);
\draw [edge] (6) to [bend left] (1);
\end{scope}

\end{tikzpicture}
\end{center}
\caption{On the top, an example of a strip $S$ with bounding paths $P$ and $Q$. 
On the bottom, an example of a closed strip (left) and an example of a M\"obius
strip (right) obtained from $S$.}
\end{figure}

Allow us to discuss some particular cases of strips. Suppose that one of the
bounding paths of a strip $S$ is trivial, i.e., it is a path on one vertex. In
this case, the closed strip obtained from $S$ is a wheel. An \textit{even strip}
is a strip with an even number of triangles; otherwise we say it is an
\textit{odd strip}.  Similarly, an \textit{even} closed (resp.\ M\"obius) strip
is a closed (resp.\ M\"obius) strip obtained from an even strip;  otherwise we
say it is an \textit{odd} closed (resp.\ M\"obius) strip. It is not hard to
notice that the number of triangles in a strip $S$ equals the number of
$PQ$-edges minus $1$ (where $P$ and $Q$ are the bounding paths of $S$). Thus,  a
closed (resp.\ M\"obius) strip with bounding paths $P$ and $Q$ is even if and
only if there an even number of $PQ$-edges.

\begin{lemma}
\label{lem:noeven-odd}
  Let $G$ be a graph. If there is a homomorphism $S\to G$ where $S$ is an odd
  closed strip or an even M\"obius strip, then $G$ does not admit a $T_3$-free
  orientation.
\end{lemma}

\begin{proof}
  By Proposition~\ref{homduality}, it suffices to show that neither odd closed
  strips nor even M\"obius strips admit a $T_3$-free orientation. Consider a
  closed strip $S$ with bounding paths $P$ and $Q$, and let $e_0, \dots,
  e_{n-1}, e_0$ be the $PQ$-edges indexed according to a clockwise ordering.
  Suppose that $S$ admits  a $T_3$-free orientation $S'$. Since all triangles in
  $S$ must be oriented cyclically in $S'$, if $e_i$ is oriented from $P$ to $Q$,
  then $e_{i+1}$ must be oriented from $Q$ to $P$ (indices taken modulo $n$).
  Inductively, $e_0$ forces an orientation of $e_{n-1}$, which by the previous
  argument, must be opposite to the orientation of $e_0$. These restrictions are
  compatible if and only if $S$ is an even closed strip. Thus, odd closed strips
  do no admit a $T_3$-free orientation. With similar arguments we see that a
  M\"obius strip is $T_3$-free orientable if and only if it is an odd M\"obius
  strip. The claim follows.
\end{proof}

Our characterization of $T_3$-free orientable graphs asserts that the converse
implication of Lemma~\ref{lem:noeven-odd} holds. To do so, it will be convenient
to describe homomorphisms from closed strips to a graph $G$ by means of certain
sequences of edges in $G$.

Notice that a strip $S$ with bounding paths $P$ and $Q$, can be described by its
sequence of $PQ$-edges (ordered from left to right), and by indicating for each
of these edges which end vertex belongs to $P$ and which to $Q$. To be precise,
we represent a strip $S$ as a sequence of edges $p_1q_1,\dots, p_nq_n$ with the
following properties: The sets  $\{p_1,\dots, p_n\}$ and $\{q_1,\dots q_n\}$
induce two disjoint paths; the intersection of $\{p_i,q_i\}$ and
$\{p_{i+1},q_{i+1}\}$ is $\{p_i\}$ or $\{q_i\}$; and either $p_i = p_{i+1}$ (so
$p_i$ is considered to be a pivot) or $p_ip_{i+1}$ is an edge. For instance, the
sequence $ax,~ay,~by,~cy,~cz$ represents the following strip.

\begin{center}
\begin{tikzpicture}

\begin{scope}[xshift=0cm, scale=0.5]
\node[]  at (-4,0){$P$};
\node[]  at (-4,2){$Q$};

\node [vertex, label = -90:\footnotesize{$a$}] (1) at (-3,0){};
\node [vertex, label = 90:\footnotesize{$x$}] (2) at (-3,2){};
\node [vertex, label = 90:\footnotesize{$y$}] (3) at (-1,2){};
\node [vertex, label = -90:\footnotesize{$b$}] (4) at (-1,0){};
\node [vertex, label = -90:\footnotesize{$c$}] (5) at (1,0){};
\node [vertex, label = 90:\footnotesize{$z$}] (6) at (1,2){};

\foreach \from/\to in {1/2,2/3,3/1,1/4,3/4,4/5,3/5,3/6,5/6}
\draw [edge] (\from) to (\to);
\end{scope}
\end{tikzpicture}
\end{center}

This representation of strips by means of edge sequences, provides a simple way
of describing homomorphisms from strips to graphs. Consider a graph $G$ and let
$p_1q_1,\dots, p_nq_n$ be a sequence of edges of $G$. This sequence defines a
homomorphism from a strip $S$ to $G$ if for every $i\in\{1,\dots, n-1\}$ the
following statements hold:
\begin{enumerate}
	\item  Either $p_i = p_{i+1}$ and $q_iq_{i+1}\in E(G)$, or $q_i = q_{i+1}$
	  and $p_ip_{i+1}\in E(G)$.
	\item If $p_n = p_1$ and $q_n = q_1$, then the sequence defines a homomorphism
	  from a closed strip to $G$. 
	\item If $p_n = q_1$ and $q_n = p_1$, then the sequence defines a homomorphism
	  from a M\"obius strip to $G$. 
\end{enumerate}
Moreover, the parity of the corresponding strip is the same as the parity of
$n-1$ (the length of the sequence).

Recall that $D^+ = (V^+,A^+)$ denotes the constraint digraph defined in
Section~\ref{sec:Algorithm}. By definition of $A^+$, it follows that if
$F=\{T_3\}$, then every arc in $A^+$ is symmetric. Thus, we may think of $D^+$
as a graph, and so, for any graph $G$ we denote by $G^+$ the constraint graph of
$G$ and $\{T_3\}$. Now, we observe that paths in $G^+$ translate to
homomorphisms of strips to $G$. To do so, we will use the previous description
of homomorphism from strips.

\begin{lemma}
\label{lem:hom-strips}
  Consider a graph $G$ and a pair  $xy$ and $zw$ of edges in $G$. If there is an
  $(x,y)(z,w)$-path of  even (resp.\ odd) length in $G^+$, then there is a
  sequence $p_1q_1,\dots, p_nq_n$ of edges of $G$ that defines a homomorphism
  from some odd (resp.\ even) strip to $G$. Furthermore, we can choose $p_1 =
  x$, $q_1 = y$, and $p_n = z$ and $q_n = w$ (resp.\ $p_n = w$ and $q_n = z$).
\end{lemma}

\begin{proof}
  We proceed by induction over the length of the  $(x,y)(z,w)$-path. Regarding
  the furthermore statement, in this paragraph we show that $\{p_1,q_1\} =
  \{x,y\}$ and $\{z,w\} = \{p_n,q_n\}$ --- we take care of the vertex equalities
  in the second paragraph. The base case is when the $(x,y)(z,w)$-path is an
  edge. Thus, by definition of $E^+$, the set vertices $\{x,y,z,w\}$ induces a
  triangle and the claim is obvious. Consider now an $(x,y)(z,w)$-path $W$ of
  length at least two, and let $(a,b)$ be the vertex before $(z,w)$ in $W$. So,
  there is an odd (resp.\ even) sequence $S = p_1q_1, \dots, p_{n-1}q_{n-1}$ of
  edges of $G$, such that $\{p_1,q_1\} = \{x,y\}$ and $\{p_{n-1},q_{n-1}\} =
  \{a,b\}$. Since there is an edge $(a,b)(z,w)$ in $G^+$, then $\{a,b,z,w\}$
  induces a triangle in $G$, so $\{a,b\}\cap \{z,w\} = \{v\}$. We extend the
  previous sequence $S$ to an even (resp.\ odd) sequence $S' = p_1q_1, \dots,
  p_{n-1}q_{n-1}, p_nq_n$ where $p_n$ and $q_n$ are defined depending on the
  value of $v$. Let $u$ be the unique vertex in $\{z,w\}\setminus\{v\}$. If $v =
  p_{n-1}$ then $p_n = v$ and $q_n = u$; otherwise $p_n = u$ and $q_n = v$. The
  fact that $S'$ defines a homomorphism from an even (resp.\ odd) strip follows
  from the choice of $p_n$ and $q_n$, and from the fact that  $\{a,b,z,w\}$
  induces a triangle (and from the induction hypothesis).

  To prove the furthermore statement, first notice that we can choose, without
  loss of generality, $p_1 = x$ and $q_1 = y$. Given this choice, we follow a
  similar inductive argument as above to verify that $p_n = z$ and $q_n = w$
  (resp.\ $p_n = w$ and $q_n = z$). Let $W$ be the $(x,y)(z,w)$-path in $G^+$,
  and $S$ the constructed sequence of edges of $G$. By construction of $S$, for
  each edge $p_iq_i$ in $S$, there is a vertex $v_i$ in $W$ representing one
  orientation of $p_iq_i$, i.e., $v_i$ equals $(p_i,q_i)$ or equals $(q_i,p_i)$.
  With an inductive argument we can notice that the orientation of $p_iq_i$
  represented by $v_i$ depends on the parity of $i$, i.e., the first edge is
  represented by the orientation from $p_1$ to $q_1$ (by assumption), the second
  one from $q_2$ to $p_2$ (by definition of $E^+$), and inductively $v_i =
  (p_i,q_i)$ when $i$ is odd, and $v_i = (q_i,p_i)$ when $i$ is even (also by
  definition of $E^+$). Finally, since $W$ is an $(x,y)(z,w)$-path and its
  length $\ell(W)$ is $n-1$, then $(z,w) = (q_n,p_n)$ if $\ell(W)$ is odd, and
  $(z,w) = (p_n,q_n)$ if $\ell(W)$ is even. 
\end{proof}

To prove our main result, recall that $G$ admits a $T_3$-free orientation if and
only if for each edge $xy\in E(G)$ the vertices $(x,y),(y,x)\in V^+$ are in
different connected components of $G^+$ (Theorem~\ref{thm:algorithm1}).
 
\begin{theorem}
\label{thm:T3-free}
  Let $\mathcal{F}$ be the set of all odd closed strips and even M\"obius
  strips. For a graph $G$, the following statements are equivalent:
  \begin{enumerate}
    \item $G$ admits a $T_3$-free orientation.
    \item There is no homomorphism $F\to G$ whenever $F\in \mathcal{F}$.
    \item There is no homomorphism from an odd closed strip to $G$.
    \item There is no homomorphism from an even M\"obius strip to $G$.
  \end{enumerate}
\end{theorem}

\begin{proof}
  Clearly, the second statement is equivalent to the conjunction of the third
  and fourth statements. So, by showing that the third and fourth statements are
  equivalent, we will prove that the second, third and fourth statements are
  equivalent.  Also, by Lemma~\ref{lem:noeven-odd}, the first statement implies
  the remaining statements. To conclude the proof we will show that the second
  statement implies the first one.

  To prove the equivalence between the last two statements, it suffices to show
  that for any odd closed (resp.\ even M\"obius) strip $S$, there is an even
  M\"obius (resp.\ odd M\"obius) strip $S'$ such that $S'\to S$. Suppose that
  $S$ is an odd closed strip, and let $S_0$ be an odd strip such that $S$ is a
  quotient of $S_0$. Let $P = p_1\dots p_n$ and $Q = q_1\dots q_m$ be  the
  bounding paths of $S_0$. Recall that the region between $P$ and $Q$ is
  triangulated in $S_0$. This implies that either $p_{n-1}q_m\in E(S_0)$ or
  $q_{m-1}p_n\in E(S_0)$; without loss of generality we assume that
  $p_{n-1}q_m\in E(S_0)$. Consider the even strip $S_1$ with bounding paths $R =
  r_1\dots r_n$ and $T = t_1\dots t_{m+1}$ with the following adjacencies. For
  each $i\in\{1,\dots, n-1\}$ and $j\in \{1,\dots, m-1\}$ there is an edge
  $r_it_j$ if and only if there is an edge $p_iq_j$. The remaining $RT$-edges
  are $r_nt_{m-1}$, $r_nt_m$ and $r_nt_{m+1}$. Clearly, $S_1$ has exactly one
  more triangle than $S_0$, so the M\"obius strip $S'$ defined by $S_1$ is an
  odd M\"obius strip. To see that $S'\to S$, first consider the mapping
  $\varphi\colon V(S_1)\to V(S_0)$ defined by $r_i\mapsto p_i$ and $t_j\mapsto
  q_j$ for every $i\in\{1,\dots, n-1\}$ and $j\in\{1,\dots m-1\}$, and
  $r_n\mapsto q_m$, $t_m\mapsto p_{n-1}$ and $t_{m+1}\mapsto p_n$. We illustrate
  this mapping and construction as follows.
\begin{center}
\begin{tikzpicture}
 
\begin{scope}[xshift=3cm, scale=0.5]
\node[]  at (2,0){$P$};
\node[]  at (2,2){$Q$};

\node [vertex, label = -90:\footnotesize{$p_1$}] (1) at (-4.5,0){};
\node [vertex, label = 90:\footnotesize{$q_1$}] (2) at (-4.5,2){};
\node [] (l) at (-4.5,1){};

\node [vertex, label = 90:\footnotesize{$q_{m-1}$}] (3) at (-1,2){};
\node [vertex, label = -90:\footnotesize{$p_{n-1}$}] (4) at (-1,0){};
\node [] (m) at (-1,1){};
\node [vertex, label = -90:\footnotesize{$p_n$}] (5) at (1,0){};
\node [vertex, label = 90:\footnotesize{$q_m$}] (6) at (1,2){};

\foreach \from/\to in {1/2, 3/4,4/5,4/6,3/6,5/6}
\draw [edge] (\from) to (\to);

\foreach \from/\to in {l/m, 1/4, 2/3}
\draw [edge, dotted] (\from) to (\to);

\end{scope}
\node[] (left) at (-2.5,0.5){};
\node[] (right) at (0.5,0.5){};
\node[]  at (-1,1){\tiny{$t_m\mapsto p_{n-1}$}};
\node[]  at (-1,0.2){\tiny{$r_n\mapsto q_m$}};
\node[]  at (-1,0.7){\tiny{$t_{m+1}\mapsto p_n$}};
\draw [arc, thin] (left) to (right);
\begin{scope}[xshift=-4cm, scale=0.5]
\node[]  at (-6,0){$R$};
\node[]  at (-6,2){$T$};
\node[] (l) at (-4.5,1){};

\node [vertex, label = -90:\footnotesize{$r_1$}] (1) at (-4.5,0){};
\node [vertex, label = 90:\footnotesize{$t_1$}] (2) at (-4.5,2){};

\node [vertex, label = 90:\footnotesize{$t_{m-1}$}] (3) at (-1,2){};
\node [vertex, label = -90:\footnotesize{$r_{n-1}$}] (4) at (-1,0){};
\node[] (m) at (-1,1){};

\node [vertex, label = -90:\footnotesize{$r_n$}] (5) at (1,0){};
\node [vertex, label = 90:\footnotesize{$t_m$}] (6) at (1,2){};

\node [vertex, label = 90:\footnotesize{$t_{m+1}$}] (7) at (3,2){};

\foreach \from/\to in {l/m, 1/4, 2/3}
\draw [edge, dotted] (\from) to (\to);

\foreach \from/\to in {1/2,3/4,4/5,3/5,3/6,5/6, 5/7, 6/7}
\draw [edge] (\from) to (\to);
\end{scope}
\end{tikzpicture}
\end{center}
  By definition of the edge set of $S_1$, the mapping $\varphi$ is a
  homomorphism from $S_1$ to $S_0$.  Furthermore, notice that $S'$ is defined
  from $S_1$ by identifying $r_n$ and $t_1$, and identifying $t_{m+1}$ and
  $r_1$. Also, $S$ is defined from $S_0$ by identifying $p_n$ and $p_1$, and
  $q_m$ and $q_1$. It is not hard to notice that $\varphi$ commutes with these
  vertex identifications. Thus, the mapping $\varphi$ factors to a homomorphism
  $S\to S'$. This shows that for every odd closed strip $S$, there is an even
  M\"obius strip $S'$ such that $S'\to S$. With a similar construction one can
  show that for every even M\"obius strip $S'$, there is an odd closed strip $S$
  such that $S\to S'$. Therefore, the third and fourth statements are
  equivalent.

  To conclude the proof we show that the second statement implies the first one.
  To do so, we prove the contrapositive statement. Suppose that $G$ does not
  admit a $T_3$-free orientation. So, by Lemma~\ref{lem:hom-strips}, there is an
  $(x,y)(y,x)$-path in $G^+$ for some edge $xy$ of $G$. Thus, applying
  Lemma~\ref{lem:hom-strips} to $(x,y)$ and $(y,x)$ yields an edge sequence
  $p_1q_1,\dots, p_nq_n$ of $G$ such that $p_1 = x$ and $q_1 = y$, and either
  $p_n = y$ and $q_n = x$ if $n$ is even, or $p_n = x$ and $q_n = y$. So, if the
  length of the edge sequence $S$ is even, then $S$ defines a homomorphism from
  an even M\"obius strip to $G$, and if the length is odd, then $S$ defines a
  homomorphism from an odd closed strip to $G$. Therefore, $G$ does not admits a
  $T_3$-free orientation then $F\to G$ for some $F\in \mathcal{F}$. This
  concludes the proof.
\end{proof}

\section{Proof of Theorem~\ref{thm:main}}
\label{sec:proof}

We first prove that for each class $\mathcal{C}$ listed in
Theorem~\ref{thm:main}, there is a set of oriented graphs on three vertices $F$,
such that $\mathcal{C}$ is the class of $Forb_e(F)$-graphs. Then, we show that
for each set $F$ of oriented graphs on three vertices, there is a class
$\mathcal{C}$ listed in Theorem~\ref{thm:main}, such that
$\mathcal{U}_{Forb_e(F)}$ equals $\mathcal{C}$, equals the intersection of
$\mathcal{C}$ and triangle-free graphs, or equals the intersection of
$\mathcal{C}$ and complete-multipartite graphs.

\subsection{Part 1}

We list the graphs in the same order as in Theorem~\ref{thm:main}. In each case
we provide the set of forbidden oriented graphs, together with the corresponding
reference. As mentioned before, we have no characterization for two of these
classes. 

\noindent
1. \textbf{Perfectly orientable graphs.}
By definition \cite{skrienJGT6}, these graphs are $Forb_e(B_1)$-graphs, and
equivalently $Forb_e(B_2)$-graphs.\\

\noindent
2. \textbf{Comparability graphs.}
These graphs are $Forb_e(B_3)$-graphs \cite{skrienJGT6}. This class also
corresponds to $Forb_e(B_3,\overrightarrow{C}_3)$
(definition).\\

\noindent
3. \textbf{Odd closed strip hom.-free graphs.}
These graphs are $Forb_e(T_3)$-graphs (Theorem~\ref{thm:T3-free}).\\

\noindent
4. \textbf{Disjount union of proper circular-arc graphs.}
These graphs are $Forb_e(B_1,B_2)$-graphs \cite{skrienJGT6}
--- see Table~\ref{Tab:skrien}.\\

\noindent
5. \textbf{Trivially perfect graphs.}
These graphs are $Forb_e(B_1,B_2)$-graphs \cite{skrienJGT6} --
see Table~\ref{Tab:skrien}.\\

\noindent
6. \textbf{Transitive-perfectly orientable graphs}
By definition, these graphs are $Forb_e(B_1,\overrightarrow{C}_3)$-graphs,
and equivalently $Forb_e(B_2,\overrightarrow{C}_3)$-graphs.\\

\noindent
7. \textbf{Disjoint union of unicyclic graphs.}
These graphs are $Forb_e(B_1,T_3)$-graphs, and equivalently
$Forb_e(B_2,T_3)$-graphs (Proposition~\ref{B1T3-free}).\\

\noindent
8. \textbf{Disjoint union of triangle-free unicyclic graphs.}
These graphs are $Forb_e(B_1,T_3, \overrightarrow{C}_3)$-graphs, and
equivalently $Forb_e(B_2,T_3, \overrightarrow{C}_3)$-graphs
(Direct implication of Proposition~\ref{B1T3-free}).\\

\noindent
9. \textbf{3-colourable comparability graphs.}
These graphs are $Forb_e(B_3,T_3)$-graphs
(Proposition~\ref{B3T3-free}).\\

\noindent
10. \textbf{Triangle-free graphs.}
These graphs are $Forb_e(\overrightarrow{C}_3,T_3)$-graphs
(trivial).\\

\noindent
11. \textbf{Clusters.}
These graphs are $Forb_e(B_1,B_2,B_3)$-graphs (trivial).\\

\noindent
12. \textbf{Disjoint union of proper Helly circular-arc graphs.}
These graphs are $Forb_e(B_1,B_2,\overrightarrow{C}_3)$-graphs
(Proposition~\ref{B1B2C3-free}).\\

\noindent
13. \textbf{Disjoint union of triangle-free proper circular-arc graphs.}
These graphs are $Forb_e(B_1,B_2,T_3,\overrightarrow{C}_3)$-graphs
(Implication of Lemma~\ref{lem:intersection} and \cite{skrienJGT6}
-- see Table~\ref{Tab:skrien}).\\

\noindent
14. \textbf{Disjoint unions of paths and cycles.}
These graphs are $Forb_e(B_1,B_2,T_3)$-graphs
(Proposition~\ref{B1B2T3-free}).\\

\noindent
15. \textbf{Disjoint unions of paths and cycles but no triangles.}
These graphs are $Forb_e(B_1,B_2,T_3, \overrightarrow{C}_3)$-graphs
(Direct implication of Proposition~\ref{B1B2T3-free}).\\

\noindent
16. \textbf{Disjoint union of stars and triangles.}
These graphs are $Forb_e(B_1,B_3,T_3)$-graphs
(Proposition~\ref{B1B3T3-free}).\\

\noindent
17. \textbf{Star Forests.}
These graphs are $Forb_e(B_1,B_3,T_3, \overrightarrow{C}_3)$-%
graphs (Corollary~\ref{B1B3C3T3-free}).\\

\noindent
18. \textbf{Stars and empty graphs.}
These graphs are $Forb_e(B_1,B_3,T_3, \overrightarrow{C}_3, T_1+T_2)$-%
graphs (Lemma~\ref{lem:intersection} and Corollary~\ref{B1B3C3T3-free}).\\

\noindent
19. \textbf{Matchings with isolated vertices.}
These graphs are $Forb_e(B_1,B_2,B_3,T_3, \overrightarrow{C}_3)$-graphs 
(trivial).\\

\noindent
20. \textbf{Empty graphs and $\mathbf{K_2}$.}
These graphs are
$Forb_e(B_1,B_2,B_3,T_3,\overrightarrow{C}_3, T_1+T_2)$-graphs (trivial).\\

\noindent
21. \textbf{Bipartite graphs.}
These graphs are $Forb_e(B_3,T_3, \overrightarrow{C}_3)$-graphs 
(Proposition~\ref{B3C3T3-free}).\\

\noindent
22. \textbf{Complete bipartite graphs.}
These graphs are $Forb_e(\{T_1+ T_2,T_3,
\overrightarrow{C}_3)$-graphs (trivial).\\

\noindent
23. \textbf{Complete 3-partite graphs.}
These graphs are $Forb_e(T_1+ T_2,T_3)$-graphs
(Corollary~\ref{K1K2T3-free}).\\

\noindent
24. \textbf{$\mathbf{K_{2,3}}$-free complete multipartite graphs.}
These graphs are $Forb_e(B_1,T_1+ T_2)$-graphs, or equivalently
$Forb_e(B_2,T_1+ T_2)$-graphs (Proposition~\ref{B1T1+T2-free}).\\

\noindent
25. \textbf{Complete multipartite graphs.}
These graphs are $Forb_e(T_1+ T_2)$-graphs
(Lemma~\ref{lem:intersection}).\\

\noindent
26. \textbf{All graphs.}
These graphs are $Forb_e(\overrightarrow{C}_3)$-graphs
(every graph admits an acyclic orientation).\\

\noindent
\textbf{Intersection with complete multipartite graphs.}
So far, we have shown that each class $\mathcal{C}$ listed in
Theorem~\ref{thm:main}, is a class of $Forb_e(F)$-graphs for some finite set $F$
of non-empty oriented graphs on three vertices. The fact that the intersection
of $\mathcal{C}$ and complete multipartite graphs is a class of
$Forb_e(F)$-graphs (for some finite set $F$ of non-empty oriented graphs on
three vertices) follows from Lemma~\ref{lem:intersection}.

\subsection{Part 2}

We present this part of the proof as a series of tables. The leftmost column of
each table lists graphs in the forbidden set $F$; the mid-column contains the
name of the class $\mathcal{U}_{Forb_e(F)}$ and the corresponding list number of
Theorem~\ref{thm:main} --- an asterisks means that the class is finite; and the
last column contains the corresponding reference. We begin with those sets that
contain exactly one graph.

\begin{table}[ht!] 
\begin{center}
    \begin{tabular}{| c | l | l |}
    \hline
    Oriented graphs in $F$ & $Forb_e(F)$-graphs & Reference \\ \hline
    $B_1$ & 1. Perfectly orientable graphs & Definition \cite{skrienJGT6} \\
    \hline
    $B_2$ & 1. Perfectly orientable graphs & Definition \cite{skrienJGT6} \\
    \hline
    $B_3$ & 2. Comparability graphs &  Skrien \cite{skrienJGT6} \\ \hline
    $\overrightarrow{C_3}$ & 21. All graphs & Trivial \\ \hline
    $T_3$ & 3. Odd closed strip hom.\-free graphs & Theorem~\ref{thm:T3-free} \\
    \hline
    $T_1+ T_2$ &  19. Complete multipartite graphs &
    Trivial\\
    \hline
    \end{tabular}
    \caption{Sets containing one oriented graphs on three vertices.}
    \end{center}
  \end{table}

From now on, we only consider sets that do not contain $T_1+ T_2$. We will treat
these cases separately. 

\begin{table}[ht!] 
\begin{center}
    \begin{tabular}{| c | l | l |}
    \hline
    Oriented graphs in $F$ & $Forb_e(F)$-graphs & Reference \\
    \hline
    $B_1,B_2$ &  4. Proper circular-arc graphs & Skrien \cite{skrienJGT6}\\
    \hline
    $B_1,B_3$ &  5. Trivially perfect graphs & Skrien \cite{skrienJGT6}\\
    \hline
    $B_1,\overrightarrow{C_3}$ & 6. Transitive-perfectly orientable graphs &
    Definition \\
    \hline
    $B_1,T_3$ & 7. Disjoint union of unicyclic graphs &
    Proposition~\ref{B1T3-free}\\
    \hline
    $B_2,B_3$ &  5. Trivially perfect graphs & Skrien \cite{skrienJGT6}\\
    \hline
    $B_2,\overrightarrow{C_3}$ & 6. Transitive-perfectly orientable graphs &
    Definition \\
    \hline
    $B_2,T_3$ & 7. Disjoint union of unicyclic graphs &
    Proposition~\ref{B1T3-free}\\
    \hline
    $B_3,\overrightarrow{C_3}$ & 2. Comparability graphs & Definiton\\ \hline
    $B_3,T_3$ & 9. $3$-colourable comparability graphs &
    Proposition~\ref{B3T3-free}\\ \hline
    $\overrightarrow{C_3},T_3$ & 10. Triangle-free graphs & Trivial \\ \hline
    \end{tabular}
    \caption{Sets containing two oriented graphs on three vertices, but not
    $T_1 + T_2$.}
    \end{center}
  \end{table}

\begin{table}[ht!] 
\begin{center}
    \begin{tabular}{| c | l | l |}
    \hline
    Oriented graphs in $F$ & $Forb_e(F)$-graphs & Reference \\
    \hline
    $B_1, B_2, B_3$ & 11. Clusters. & Trivial. \\
    \hline
    $B_1,B_2,\overrightarrow{C_3}$ & 12. Proper Helly circular-arc graphs.&
    Proposition~\ref{B1B2C3-free}.\\
    \hline
    $B_1,B_2,T_3$ & 13. Disjoint union of paths and cycles.
    & Proposition~\ref{B1B2T3-free}.\\
    \hline
    $B_1,B_3,\overrightarrow{C_3}$ & 5. Trivially perfect graphs. &
    Proposition~\ref{B1B3C3-free}.\\
    \hline
    $B_1,B_3,T_3$ & 14. Disjoint union of triangles and stars.&
    Proposition~\ref{B1B3T3-free}.\\
    \hline
    $B_1,\overrightarrow{C_3},T_3$ & 8. Disjoint union of triangle-free
    unicyclic graphs.
    & Corollary~\ref{B1C3T3-free}.\\
    \hline
    $B_2,B_3,\overrightarrow{C_3}$ & 5. Trivially perfect graphs. &
    Proposition~\ref{B1B3C3-free}.\\
    \hline
    $B_2,B_3,T_3$ & 14. Disjoint union of triangles and stars.&
    Proposition~\ref{B1B3T3-free}.\\
    \hline
    $B_2,\overrightarrow{C_3},T_3$ & 8. Disjoint union of triangle-free
    unicyclic graphs.
    & Corollary~\ref{B1C3T3-free}.\\
    \hline
    $B_3,\overrightarrow{C_3},T_3$ & 16. Bipartite graphs. &
    Proposition~\ref{B3C3T3-free}\\
    \hline
    \end{tabular}
    \caption{Sets containing three oriented graphs on three vertices, but
    not $T_1+T_2$.}
    \end{center}
  \end{table}

\begin{table}[ht!] 
\begin{center}
    \begin{tabular}{| c | l | l |}
      \hline
      Oriented graphs in $F$ & $Forb_e(F)$-graphs & Reference \\
      \hline
      $B_1,B_2,B_3,\overrightarrow{C_3}$ & 11. Clusters & Trivial \\
      \hline
      $B_1,B_2,B_3,T_3,$ & $\ast$ $K_3,~K_2$ and $K_1$ & Trivial \\
      \hline
      $B_1,B_2,\overrightarrow{C_3}, T_3$ & 11. D.u.o.\ triangle-free proper
      circular-arc graphs & Lemma~\ref{lem:intersection} +
      Table~\ref{Tab:skrien} \\
      \hline
      $B_1,B_3,\overrightarrow{C_3}, T_3$ & Star forest &
      Corollary~\ref{B1B3C3T3-free} \\
      \hline
      $B_2,B_3,\overrightarrow{C_3}, T_3$ &Star forest &
      Corollary~\ref{B1B3C3T3-free} \\
      \hline
      $B_1,B_2,B_3,\overrightarrow{C_3}, T_3$ & Matchings with isolated vertices
      & Lemma~\ref{lem:intersection} \\
      \hline
    \end{tabular}
    \caption{Sets containing four or five oriented graphs on three vertices, but
    not $T_1+T_2$ nor both orientations of the triangle.}
    \end{center}
  \end{table}

The tables displayed in this section, show that if  $(T_1+T_2)\not \in F$ or
$|F| = 1$, then the class of $Forb_e(F)$-graphs is either finite or listed in
Theorem~\ref{thm:main}. By Lemma~\ref{lem:intersection}, if $(T_1+T_2)\in F$,
then the class $\mathcal{U}_{Forb_e(F)}$ is the intersection of
$\mathcal{U}_{Forb_e(F-(T_1+T_2)}$ and complete multipartite graphs. Thus, the
claim of Theorem~\ref{thm:main} holds.
\pagebreak


\section{Complete multipartite graphs}
\label{sec:completemultipartite}

For the sake of completeness, we comment on the intersection of classes listed
in Theorem~\ref{thm:main} and complete multipartite graphs. 

\begin{proposition}
\label{B1C3T+T2-free}
  For a complete multipartite graph, the following statements are equivalent:
  \begin{enumerate}
    \item $G$ admits a $\{B_1,\overrightarrow{C}_3,T_1+T_2\}$-free orientation.
    \item $G$ admits a $\{B_2,\overrightarrow{C}_3,T_1+T_2\}$-free orientation.
    \item $G$ is a transitive-perfectly orientable graph. 
    \item $G$ is a complete graph, a complete graph minus two non-incident
      edges, or a complete split graph.
    \item $G$ is a $\{K_{2,3},K_{2,2,2}\}$-free complete
    multipartite graph.
  \end{enumerate}
\end{proposition}

\begin{proof}
  The equivalence between first three items follow from
  Proposition~\ref{B1B2C3-free} and from Lemma~\ref{lem:intersection}. The last
  two statements are evidently equivalent. It is straighforward to find a
  $\{B_1,\overrightarrow{C}_3,T_1+T_2\}$-free orientation of a graph described
  in item 4, so the fourth statement implies the first three. On the other hand,
  it is not hard to notice that neither $K_{2,3}$ nor $K_{2,2,2}$ admit a
  $\{B_1,\overrightarrow{C}_3,T_1+T_2\}$-free orientation, so the first
  statement implies the last two. This concludes the proof.
\end{proof}

\begin{proposition}
\label{B1B2C3T+T2-free}
  For a complete multipartite graph, the following statements are equivalent:
  \begin{enumerate}
    \item $G$ admits a $\{B_1,B_2,\overrightarrow{C}_3,T_1+T_2\}$-free
      orientation.
    \item $G$ is a proper Helly circular-arc graph. 
    \item $G$ is an empty graph, a complete graph, a complete graph minus an
      edge, or $C_4$.
    \item $G$ is a $\{K_{1,3}, K_{1,2,2}\}$-free complete multipartite graph.
  \end{enumerate}
\end{proposition}

\begin{proof}
  The equivalence between the first two items follows from
  Proposition~\ref{B1B2C3-free} and Lemma~\ref{lem:intersection}. The last two
  statements are evidently equivalent. It is immediate to find a
  $\{B_1,B_2,\overrightarrow{C}_3,T_1+T_2\}$-free orientation of a graph listed
  in the third item. So,  the third statement implies the first two. Finally, it
  is not hard to notice that the $K_{1,3}$ is not a proper circular-arc graph,
  and $K_{1,2,2}$ is the $4$-wheel which is not a proper Helly circular-arc
  graph \cite{linDAM2013}. Thus, the second statement implies the last two. The
  claim follows.
\end{proof}

\begin{proposition}
\label{B1B2T+T2-free}
  For a complete multipartite graph, the following statements are equivalent:
  \begin{enumerate}
    \item $G$ is a proper-circular arc graph.
    \item $G$ admits a $\{B_1,B_2,T_1+T_2\}$-free orientation. 
    \item $G$ is an empty graphs, a complete graph, a complete graph minus an
      edge, or a complete graph minus two non-incident edges. 
    \item $G$ is a $\{K_{1,3},K_{2,2,2}\}$-free complete multipartite graph.
  \end{enumerate}
\end{proposition}

\begin{proof}
  The equivalence between the first two items follows from
  Table~\ref{Tab:skrien} and Lemma~\ref{lem:intersection}. The last two
  statements are evidently equivalent. On the one hand, it is immediate to find
  a $\{B_1,B_2,T_1+T_2\}$-free orientation of an empty graph, a complete graph,
  a complete graph minus and edge or a complete graph minus two non-incident
  edges. On the other one, it is not hard to see that neither $K_{1,3}$ nor
  $K_{2,2,2}$ admit such an orientation. So, the equivalence between the four
  statements holds.
\end{proof}

\begin{proposition}
\label{B1B3C3T+T2-free}
  For a graph $G$, the following statements are equivalent:
  \begin{enumerate}
    \item $G$ is a trivially perfect complete multipartite graph.
    \item $G$ admits a $\{B_1,B_3,T_1+T_2\}$-free orientation. 
    \item $G$ admits a $\{B_1,B_3,\overrightarrow{C}_3,T_1+T_2\}$-free
      orientation.
    \item $G$ is a $C_4$-free complete multipartite graph. 
    \item $G$ is an empty graph, a complete graph or a complete graph minus an
      edge.
  \end{enumerate}
\end{proposition}

\begin{proof}
  The equivalence between the first three items follows from
  Lemma~\ref{lem:intersection} and Proposition~\ref{B1B3C3-free}. The
  equivalence between the fourth and fifth statements is immediate. Finally,
  recall that trivially perfect graphs are $\{C_4,P_4\}$-free graphs
  \cite{golumbicDM24}. Thus, since every complete multipartite graph is
  $P_4$-free, we conclude that the first and fourth statements are equivalent,
  which concludes the proof.
\end{proof}

\begin{theorem}
  The following classes are all infinite families of $Forb_e(F)$-graphs, where
  $F$ is a set of non-empty oriented graphs on three vertices and $Forb_e(F)$ is
  a subclass of complete multipartite graphs.
  \begin{multicols}{2}
  \begin{enumerate}
    \item Complete multipartite graphs.
    \item Complete $3$-partite graphs.
    \item Complete bipartite graphs.
    \item $K_{2,3}$-free complete multipartite graphs.
    \item $\{K_{2,3}, K_{2,2,2}\}$-free complete multipartite graphs.
    \item $\{K_{2,3}, K_{1,2,2}\}$-free complete multipartite graphs.
    \item $C_4$-free complete multipartite graphs.
    \item Complete graphs and empty graphs.
    \item Empty graphs, stars, $C_3$ and $C_4$.
    \item Empty graphs, stars and $C_4$.
    \item Empty graphs, stars and $C_3$.
    \item Empty graphs and stars.
    \item Empty graphs and finitely many graphs.
  \end{enumerate}
  \end{multicols}
\end{theorem}

\begin{proof}
  We show that the intersection of complete multipartite graphs and each class
  is the list of Theorem~\ref{thm:main} is listed above. We proceed according to
  the listing order in Theorem~\ref{thm:main}.

  \textit{(Thm~\ref{thm:main}.1)} The intersection of perfectly orientable
  graphs and complete multipartite graphs are $K_{2,3}$-free complete
  multipartite graphs (Proposition~\ref{B1T1+T2-free}).

  \textit{(Thm~\ref{thm:main}.2)} The intersection of comparability graphs and
  complete multipartite graphs equals the class of complete multipartite graphs
  (every complete multipartite graph is a comparability graph).

  \textit{(Thm~\ref{thm:main}.3)} The intersection of odd closed strip hom.-free
  graphs and  complete multipartite graphs equal the class of complete
  $3$-partite graphs (Corollary~\ref{K1K2T3-free}).

  \textit{(Thm~\ref{thm:main}.4)} The intersection of d.u.o.\ proper
  circular-arc graphs and  complete multipartite graphs equal the class of
  $\{K_{1,3}, K_{2,2,2}\}$-free complete multipartite graphs
  (Proposition~\ref{B1B2T+T2-free}).

  \textit{(Thm~\ref{thm:main}.5)} The intersection of trivially perfect graphs
  and  complete multipartite graphs are either complete graphs or complete
  graphs minus an edge (Proposition~\ref{B1B3C3T+T2-free}).

  \textit{(Thm~\ref{thm:main}.6)} The intersection of transitive-perfectly
  orientable graphs and  complete multipartite graphs equal the class of
  $\{K_{2,3}, K_{2,2,2}\}$-free complete multipartite graphs
  (Proposition~\ref{B1C3T+T2-free}).

  \textit{(Thm~\ref{thm:main}.7)} The intersection of d.u.o.\ of unicyclic
  graphs and  complete multipartite graphs are either empty graphs, stars, $C_3$
  or $C_4$ (immediate).

  \textit{(Thm~\ref{thm:main}.8)} The intersection of d.u.o.\ of triangle-free
  unicyclic graphs and  complete multipartite graphs are either empty graphs,
  stars or $C_4$ (immediate).

  \textit{(Thm~\ref{thm:main}.9)} The intersection of $3$-colourable
  comparability graphs and  complete multipartite graphs are complete
  $3$-partite graphs (every complete multipartite graph is a comparability
  graph).

  \textit{(Thm~\ref{thm:main}.10)} The intersection of triangle-free graphs and
  complete multipartite graphs are complete bipartite graphs (immediate).

  \textit{(Thm~\ref{thm:main}.11)} The intersection of clusters and complete
  multipartite graphs are complete graphs and empty graphs (immediate).

  \textit{(Thm~\ref{thm:main}.12)} The intersection of d.u.o.\ proper Helly
  circular-arc graphs and  complete multipartite graphs are $\{K_{1,3},
  K_{1,2,2}\}$-free complete  multipartite graphs
  (Proposition~\ref{B1B2C3T+T2-free}).

  \textit{(Thm~\ref{thm:main}.13) The intersection of d.u.o.\ triangle-free
  proper circular-arc graphs} and  complete multipartite graphs is the class
  empty graphs and some finite set (immediate from
  Proposition~\ref{B1B2T+T2-free}).

  \textit{(Thm~\ref{thm:main}.14--\ref{thm:main}.15)} The intersection of
  d.u.o.\ paths and cycles or  d.u.o.\ paths and cycles but no triangles, with
  complete multipartite graphs are either empty graphs and some finite set of
  graphs (immediate).

  \textit{(Thm~\ref{thm:main}.16)} The intersection of d.u.o.\ triangles and
  stars and  complete multipartite graphs is the class of stars and $C_3$.
  (trivial).

  \textit{(Thm~\ref{thm:main}.17--\ref{thm:main}.18)} The intersections of star
  forests and of stars and empty graphs with complete multipartite graphs is the
  class of stars and empty graphs (trivial).

  \textit{(Thm~\ref{thm:main}.19)} The intersection of matchings and isolated
  vertices with  complete multipartite are either empty graphs or $K_2$.
  (trivial).

  \textit{(Thm~\ref{thm:main}.20--\ref{thm:main}.26)} The intersections of
  either empty graphs and $K_2$, bipartite graphs, complete bipartite graphs,
  complete $3$-partite graphs, $K_{2,3}$-free complete multipartite graph,
  complete multipartite graphs or of all graphs, with complete multipartite
  graphs can be trivially described (and are listed above).
\end{proof}

\section{Conclusions}
\label{sec:conclusions}

Algorithm~\ref{alg:master} is a certifying one, i.e., given a graph $G$, it
outputs an $F$-free orientation of $G$ if it has one, or it finds and
obstruction to being a $Forb_e(F)$-graph, but these obstructions live in the
constraint digraph $D^+$, not in $G$. The proofs of Lemma~\ref{lem:hom-strips}
and Theorem~\ref{thm:T3-free}, yield a polynomial time extension of this
algorithm (in the case when $F = \{T_3\}$) that outputs an obstruction that now
lives in $G$; namely it outputs a forbidden homomorphic pre-image $W$ and a
homomorphism $\varphi \colon W \to G$. Various of the reductions to $2$-SAT are
examples of certifying algorithms that exhibit an obstruction that does not
belong to the graph $G$. A technique similar to the reverse engineering in the
proof of Lemma~\ref{lem:hom-strips} could work to find obstructions in $G$ for
other cases.

We listed all families of $Forb_e(F)$-graphs where $F$ consists of non-empty
oriented graphs on three vertices.  Finding nice characterizations of perfectly
orientable graphs and transitive-perfectly orientable graphs remain as open
problems. 

\begin{problem}
Characterize transitive-perfectly orientable graphs.
\end{problem}

We briefly observe the following structural property of transitive-perfectly
orientable graphs.

\begin{proposition}
  Every transitive-perfectly orientable graph $G$ admits a partition into two
  induced chordal graphs.
\end{proposition}

\begin{proof}
Let $G'$ be a $\{B_1,\overrightarrow{C}_3\}$-free orientation of a graph $G$. In
\cite{aboulkerEJC106}, the authors show that any
$\{B_1,\overrightarrow{C}_3\}$-free oriented graph has dichromatic number at
most $2$ (this result is also a consequence of a stronger statement found in
\cite{steinerJGT2022}). Let $U$ and $V$ be the two colour classes in such a
colouring of $G'$. Since $G'$ is $B_1$-free, and $G'[U]$ and $G'[V]$ have no
directed cycles, then the underlying induced subgraphs, $G[U]$ and $G[V]$, are
chordal graphs. The claim follows.
\end{proof}

Theorem~\ref{thm:main} together with Proposition~\ref{prop:somesets}, show that
all classes of $Forb_e(F)$-graphs can be recognized in polynomial time ---
except for transitive-perfectly orientable graphs, whose recognition complexity
remains an open problem.

\begin{theorem}
  Let $F$ be a set of oriented graphs on three vertices. If $F \neq \{B_1,
  \overrightarrow{C}_3\}$ and $F \neq \{B_2, \overrightarrow{C}_3\}$, then it is
  in $P$ to recognize $Forb_e(F)$-graphs.
\end{theorem}

Clearly, a graph $G$ admits a $\{B_1, \overrightarrow{C}_3\}$-free orientation
if and only if it admits a $\{B_2, \overrightarrow{C}_3\}$-free orientation. 

\begin{problem}
  Determine the complexity of deciding if an input graph $G$ admits a
  $\{B_1,\overrightarrow{C}_3\}$-free orientation. Equivalently, determine the
  complexity of recognizing transitive-perfectly orientable graphs.
\end{problem}

As a final conclusion, let us to see how Skrien's work \cite{skrienJGT6} and
this work relate to characterizations through forbidden ordered patterns. An
\textit{ordered pattern} is a graph, $G$, together with a linear ordering of its
vertices. Similar to the procedure followed in \cite{skrienJGT6} and in this
work, one can fix a finite set of ordered patterns, $P$, and characterize those
graphs that admit a $P$-free ordering. For instance, if $P$ is the singleton
$\{(\{1 \le 2\le 3\},\{12,13\})\}$, then a graph $G$ admits a $P$-free ordering
if and only if $G$ is chordal. Our work of Section~\ref{sec:small} is similar
(but not as thorough and complete) to \cite{feuilloleyJDM35}, where Feuilloley
and Habib characterize all families of graphs defined by admitting a $P$-free
ordering for any set of oriented graphs on three vertices, $P$. On the other
hand, the algorithm exhibited in Section~\ref{sec:Algorithm}, was motivated by
\cite{hellESA2014}, where Hell, Mohar and Rafiey propose a master algorithm that
determines if an input graph, $G$, admits a $P$-free ordering, for any fixed set
of ordered patterns on three vertices.

Given a set $F$ of orientations of $P_3$, Skrien studied classes of graphs that
admit an $F$-free acyclic orientation~\cite{skrienJGT6}. So, a possible problem
one may think of, is to extend Skrien's work as we did in this manuscript. Turns
out that this has been indirectly solved in \cite{feuilloleyJDM35} and
\cite{hellESA2014}. Consider an acyclic oriented graph $H'$ with underlying
graph $H$. Denote by $P_{H'}$ the set of all ordered patterns, $(H,\le)$, such
that for any edge $xy\in E(H)$ it holds that $(x,y)\in A(H')$ if and only if
$x\le y$.   Given a set of acyclic oriented graphs, $F$, we denote by $P_F$ the
union of all sets, $P_{H'}$, where $H'\in F$. It is not hard to observe that the
following observation holds.

\begin{observation}
\label{obs:final}
  Let $F$ be a set of acyclic oriented graphs, and let $P_F$ be the set of
  ordered patterns defined above. A graph $G$ admits an acyclic $F$-free
  orientation if and only if it admits a $P_F$-free ordering.
\end{observation}

In light of this observation, if $F$ is a set of acyclic oriented graphs on
three vertices, then the class of graphs that admit an $F$-free acyclic
orientation is characterized in \cite{feuilloleyJDM35}. Moreover, due to the
algorithm of Hell et al.\ \cite{hellESA2014}, the following statement follows.

\begin{proposition}
  Let $F$ be any set of oriented graphs on three vertices. Recognizing if an
  input graph admits an $F$-free acyclic orientation can be done in polynomial
  time.
\end{proposition}

\begin{proof}
  It follows directly from Observation~\ref{obs:final} and Corollary 1 in
  \cite{hellESA2014}.
\end{proof}


\end{document}